\documentclass[a4paper,10pt]{article}
 
\usepackage[utf8]{inputenc}
\usepackage{amssymb} 
\usepackage{amsmath}
\usepackage{amsthm}
\usepackage{mathrsfs}
\usepackage{mathtools}
\usepackage{stmaryrd}
\usepackage{bm}
\usepackage{enumitem}
\theoremstyle{plain} 
\newtheorem{theorem}{Theorem}
\newtheorem{theoremA}[theorem]{Theorem}
\newtheorem{theoremB1}[theorem]{Theorem}
\newtheorem{theoremB2}[theorem]{Theorem}
\newtheorem{theoremB3}[theorem]{Theorem}

\newtheorem{theoremD}[theorem]{Theorem}
\newtheorem{theoremE}[theorem]{Theorem}
\newtheorem{theoremF}[theorem]{Theorem}
\newtheorem*{theorem*}{Theorem}
\newtheorem{proposition}[theorem]{Proposition}
\newtheorem{lemma}[theorem]{Lemma}

\theoremstyle{definition}
\newtheorem*{remark}{Remark}
\newtheorem{example}[theorem]{Example}
\newtheorem*{definition}{Definition}
\newtheorem*{notation}{Notation}

\def\mbb #1{\mathbb{#1}}

\def\mrm #1{\mathrm{#1}}
\def\Cal #1{\mathcal{#1}}

\def\C{\mbb{C}}
\def\CP{\mbb{CP}}
\def\R{\mbb{R}}
\def\Z{\mbb{Z}}

\def\PGL{\mrm{PGL}}
\def\tr{\mrm{tr}}
\def\pr{\mrm{pr}}
\def\id{\mrm{id}}

\def\res{\mathrm{Res}}

\newcommand{\sminus}{\smallsetminus}

\newcommand*{\transp}[2][-3mu]{\ensuremath{\mskip1mu\prescript{\smash{\mathrm t\mkern#1}}{}{\mathstrut#2}}}%

\author{Martin Klime\v{s}}
\title{On equivalence of singularities of linear second order ordinary differential equations by point transformations}
\begin{document}
\maketitle

\begin{abstract}
\noindent
The article provides a local classification of singularities of meromorphic second order linear differential equation with respect to analytic/meromorphic linear point transformations. It also addresses the problem of determining the Lie algebra of analytic linear infinitesimal symmetries of the singularities.

\def\keywords#1{\small{\textbf{Key words:} #1}}\def\and{\ifhmode\unskip\nobreak\fi\ $\cdot$\ }
\smallskip\noindent
\end{abstract}

\section{Introduction}

Considering a second order linear differential equation (shortly LDE) on a complex domain $\Omega\subseteq\C$
\begin{equation}\label{eq:LT-ODE}
 \tfrac{d^2y}{dx^2}+a_1(x) \tfrac{dy}{dx}+a_0(x)y=0, \qquad (x,y)\in\Omega\times\C,
\end{equation}
with meromorphic coefficients $a_1(x), a_0(x)$,
one can associate to it its companion linear differential system for $v=\transp{\big(y,\tfrac{dy}{dx}\big)}$
\begin{equation}\label{eq:LT-companionsystem}
 \tfrac{dv}{dx}=\left(\begin{smallmatrix} 0 & 1 \\[3pt] -a_0 & -a_1 \end{smallmatrix}\right)v, \qquad (x,v)\in\Omega\times\C^2.
\end{equation}
And vice versa every \hbox{$2\!\times\!2$} linear differential system with meromorphic coefficients can be meromorphically transformed to the form of a companion system \eqref{eq:LT-companionsystem} by virtue of the so called ``cyclic vector lemma'' \cite[\S 2.1]{SP}.
From the point of study of the solutions there is not much distinction between meromorphic second order LDEs \eqref{eq:LT-ODE} and 
\hbox{$2\!\times\!2$} meromorphic differential  systems \eqref{eq:LT-companionsystem}. Correspondingly, the two are also interchangeable from the point of view of differential Galois theory.  
However, there is an important difference between them when it comes to the underlying space and to its transformation group, and consequently also when it comes to their Lie symmetry groups.

A meromorphic linear differential system 
\begin{equation}\label{eq:LT-system}
\tfrac{dv}{dx}=A(x)v, \qquad (x,v)\in\Omega\times\C^2,
\end{equation}
can be seen as a meromorphic connection on a (trivial) rank 2 vector bundle over $\Omega$.
As such, a natural kind of transformations of the system are the \emph{gauge transformations}
\begin{equation}\label{eq:LT-systemgauge}
 \tilde x=\phi(x),\qquad \tilde v=T(x)v,
\end{equation}
with $\phi:\Omega\to\Omega$ an analytic diffeomorphism and $T(x)$ a matrix-valued function that is either analytic and analytically invertible, or meromorphic and meromorphically invertible.
The corresponding systems for $v$ and $\tilde v$ are then called \emph{analytically}, resp. \emph{meromorphically, equivalent}.
In the literature one often considers the above transformations with $\phi=\id_\Omega$, but here we opt for the more general form \eqref{eq:LT-systemgauge}.

On the other hand, a second order LDE \eqref{eq:LT-ODE} is living as a connection on the 1-jet bundle $\Cal J^1(\Omega)$
(or more generally on a codimension 1 subbundle of the 2-jet bundle $\Cal J^2(\Omega)$, cf. \cite[§19E]{IlYa}).
A natural kind of transformations are \emph{linear point transformations}, that is the jet prolongations of the transformations
\begin{equation}\label{eq:LT-ODEpoint}
 \tilde x=\phi(x),\qquad \tilde y=t(x)y,
\end{equation}
with $\phi:\Omega\to\Omega$ analytic diffeomorphism and $t(x)$ either a non-vanishing analytic, or a non-zero meromorphic,  function.
Such transformations were considered already by Kummer \cite{Kum}, and it is known that general point transformations 
preserving the class of second (or higher) order LDEs are locally of the this form \cite{Sta}.
The corresponding equations for $y$ and $\tilde y$ are then called \emph{analytically}, resp. \emph{meromorphically, equivalent}.

For higher order differential equations the group of transformations \eqref{eq:LT-ODEpoint} is very restrictive, therefore more general transformations are often considered
such as ``Weyl transformations'' (generalized linear contact transformations) \cite{IlYa,TY}
\begin{equation*}%\label{eq:LT-ODEweyl}
 \tilde x=\phi(x),\qquad \tilde y=t_0(x)y+\ldots +t_{n-1}(x) \left(\tfrac{\partial}{\partial x}\right)^{n-1}y,
\end{equation*}
where $n$ is the order of the LDE.
Nevertheless in the case $n=2$ that we consider the dimension of the solutions space is the same as the dimension of the ambient $(x,y)$-space and locally
the linear point transformations \eqref{eq:LT-ODEpoint} are relatively rich enough to provide, at least in a generic situation, a satisfactory theory.

In this paper, we are interested in \emph{local analytic classification} of meromorphic second order LDEs \eqref{eq:LT-ODE} with respect to linear point transformations \eqref{eq:LT-ODEpoint} near a fixed point in the $x$-space placed at the origin, i.e. $\Omega=(\C,0)$.
If both coefficients $a_0(x), a_1(x)$ of \eqref{eq:LT-ODE} are analytic at $x=0$, then it is well known %since the works of Lie \cite{Lie} 
that the equation is locally analytically equivalent to 
$$ \tfrac{d^2y}{dx^2}=0.$$
We will therefore investigate only singular points, where at least one of $a_0(x),\ a_1(x)$ has a pole.

While the local analytic/meromorphic theory of meromorphic linear differential systems with respect gauge transformations fixing $x$ is well established by now, e.g. \cite{BaVa,Ba,IlYa,SP, BJL12},
and it is easily generalized to all transformation \eqref{eq:LT-systemgauge}, there appears to be very little written on the subject of local analytic/meromorphic classification of singular second-order LDEs with respect to linear point transformations.
We will describe the local analytic/meromorphic moduli space of LDEs \eqref{eq:LT-ODE} at singularities of regular or non-resonant irregular type (Sections~\ref{sec:1.2} and \ref{sec:1.4}) and describe their Lie groups of linear infinitesimal symmetries (Section~\ref{sec:1.5}).
%and conjecture the existence of a Birkhoff normal form for irreducible non-resonant irregular singularities (Section~\ref{sec:1.3}). 

\subsection{Singularities of linear differential equations}\label{sec:1.1} 
Suppose that at least one coefficient $a_i(x)$, $i=0,1$, of \eqref{eq:LT-ODE} has a pole at the origin, and denote
$m_i\in\Z$ the respective orders of the poles. 
%The number 
%$$\max\{0,m_1-1,\tfrac{m_0}{2}-1\}$$ 
%is called the \emph{Katz rank}; it corresponds to the top exponential order of a general solution.
Let 
$$\nu=\max\{m_1-1,\lceil\tfrac{m_0}{2}\rceil-1\}\geq 0;$$ 
%be the upper whole part of the Katz rank; 
we will call it the \emph{Poincar\'e rank} of the LDE.
Denoting
$$\delta_\nu:=x^{\nu+1}\tfrac{\partial}{\partial x},$$
and rewriting the equation \eqref{eq:LT-ODE} as
\begin{equation}\label{eq:LT-irregODE}
 \delta_\nu^2 y-p(x)\delta_\nu y-q(x)y=0, 
\end{equation}
where $p(x)=-x^{\nu+1}a_1(x)+(\nu+1)x^\nu$ and $q(x)=-x^{2\nu+2}a_0(x)$,
then the Poincar\'e rank $\nu\geq 0$ is the smallest integer for which $p(x)$, $q(x)$ are analytic at $0$.

The \emph{companion system} for \eqref{eq:LT-irregODE} is the system for $v=\transp{\big(y,\delta_\nu y\big)}$
\begin{equation}\label{eq:LT-irregsyst}
\delta_\nu v=\left(\begin{smallmatrix} 0 & 1 \\[3pt] q(x) & p(x) \end{smallmatrix}\right)v.
\end{equation} 
The equation \eqref{eq:LT-irregODE} can be rewritten as
\begin{equation}\label{eq:LT-Delta}
\big(\delta_\nu-\tfrac{p(x)}{2}\big)^2y=\tfrac{\Delta(x)}{4} y,\quad \text{where }\ 
\Delta:=p^2+4q-2\delta_\nu p,
\end{equation}	
and the system for $u=\transp{\big(y,(\delta_\nu-\tfrac{p}{2})y\big)}$ as
\begin{equation*}%\label{eq:LT-irregsyst1}
\delta_\nu u=\left(\begin{smallmatrix} \tfrac{p(x)}{2} & 1 \\[3pt] \tfrac{\Delta(x)}{4} & \tfrac{p(x)}{2} \end{smallmatrix}\right)v.
\end{equation*} 
An equation \eqref{eq:LT-irregODE} with $p(x)=0$ will be called \emph{trace-free}.

\begin{definition}
\begin{enumerate}[wide=0pt, leftmargin=\parindent]
\item If $\nu=0$, then the singularity is called \emph{regular} or \emph{Fuchsian} (by a theorem of Fuchs the notions of regular singularity and Fuchsian singularity coincide for LDEs \cite[Theorem 19.20]{IlYa}).
A Fuchsian singularity is \emph{strongly non-resonant} if the roots of $\lambda^2-p(0)\lambda-q(0)=0$ do not differ by an integer,\footnote{This condition is slightly stronger than the usual non-resonance condition that demands that the roots do not differ by a \emph{non-zero} integer \cite[Definition~16.12]{IlYa}.} 
i.e. if $\sqrt{\Delta(0)}\notin\Z$,
otherwise it is called \emph{resonant}. 
		%This is equivalent to the monodromy around 0 being semi-simple (diagonalisable), and to the companion system \eqref{eq:LT-irregsyst} being analytically diagonalisable (cf. \cite[Theorem...]{IlYa}). 
		
\item If $\nu>0$, then the singularity is called \emph{irregular}. 
An irregular singularity is \emph{non-resonant} if the roots of
$\lambda^2-p(0)\lambda-q(0)=0$ are distinct; i.e. if $\Delta(0)\neq 0$.

It will be called \emph{non-degenerate} if 
\begin{itemize}[nosep]
	\item either $\Delta(0)\neq 0$ (non-resonant),
	\item or $\Delta(0)=0$ and $p(0)=0$, in which case $\frac{d\Delta}{dx}(0)\neq 0$ (otherwise the Poincar\'e rank would be lower than $\nu$).
\end{itemize}
This is equivalent to say that in the coordinate $s=x^{\frac12}$ the LDE has a non-resonant irregular singularity. 
\end{enumerate}
\end{definition}

We shall prove the following result relating the formal and the analytic equivalence of singular LDEs with that of their companion systems \eqref{eq:LT-irregsyst}.

\begin{theoremA}\label{thm-A}~
Two LDEs \eqref{eq:LT-irregODE} with either a regular singularity or a non-degenerate irregular singularity at the origin are analytically, resp. formally, equivalent if and only if the companion systems \eqref{eq:LT-irregsyst} are analytically, resp. formally, equivalent.
\end{theoremA}

Let us stress that we consider equivalence by linear point transformations \eqref{eq:LT-ODEpoint} for equations, and by gauge transformations \eqref{eq:LT-systemgauge} for systems. By \emph{formal equivalence} it is meant that the diffeomorphism $\phi(x)$ and the function $t(x)$ in \eqref{eq:LT-ODEpoint}, resp. $T(x)$ in \eqref{eq:LT-systemgauge}, are formal power series of $x$.

On the other hand, for degenerate irregular singularities it is possible to have analytically inequivalent LDEs with analytically equivalent companion systems (Example~\ref{example:LT-1}).

In the following section we describe the moduli space of analytic equivalence of regular and non-degenerate irregular singularities of LDEs. In the light of Theorem~\ref{thm-A}, these results are parallel to the classical theory of singularities of linear differential systems and are a direct reformulation of it.

\subsection{Formal and analytic classification}\label{sec:1.2} 

It is easy to verify that if $y_1(x)$, $y_2(x)$ are two linearly independent solutions to \eqref{eq:LT-irregODE} then one can express $\Delta(x)$ \eqref{eq:LT-Delta} as
$$\Delta(x)=-2\Cal S_\nu\big(\tfrac{y_1}{y_2}\big)(x),\quad\text{where }\
\Cal S_\nu(f)=\delta_\nu\big(\tfrac{\delta_\nu^2f}{\delta_\nu f}\big)-\tfrac{1}{2}\big(\tfrac{\delta_\nu^2f}{\delta_\nu f}\big)^2$$
is the \emph{``Schwarzian derivative''} associated to $\delta _\nu$.
The point transformation \eqref{eq:LT-ODEpoint} relate two such LDEs by
\begin{align}
p&=\psi\cdot\tilde p\circ\phi+\tfrac{\delta_\nu\psi}{\psi}-2\tfrac{\delta_\nu t}{t},\label{eq:LT-transformationp}\\
q&=\psi^2\cdot\tilde q\circ\phi+\big[\psi\cdot\tilde p\circ\phi+\tfrac{\delta_\nu\psi}{\psi}\big]\tfrac{\delta_\nu t}{t}-\tfrac{\delta_\nu^2 t}{t},\label{eq:LT-transformationq}\\
\Delta&=\psi^2\cdot\tilde{\Delta}\circ\phi-2\delta_\nu\big(\tfrac{\delta_\nu\psi}{\psi}\big)+\big(\tfrac{\delta_\nu\psi}{\psi}\big)^2,
\label{eq:LT-transformationDelta}
\end{align}
where $\psi(x)=\tfrac{\delta_\nu\phi}{\phi^{\nu+1}}$, i.e. $\delta_\nu=\psi(x)\tilde\delta_\nu$.
In fact, the formula \eqref{eq:LT-transformationDelta} is just the transformation rule for the above ``Schwarzian derivative''
$$\Cal S_\nu(f\circ\phi)=\psi^2\cdot\Cal S_\nu(f)\circ\phi+\delta_\nu\big(\tfrac{\delta_\nu\psi}{\psi}\big)-\tfrac12\big(\tfrac{\delta_\nu\psi}{\psi}\big)^2.$$
In this notation, the operator $\Cal S_{-1}$ is  the usual Schwarzian derivative for which the transformation rule reads 
$\Cal S_{-1}(f\circ\phi)=\left(\frac{d\phi}{dx}\right)^2\!\!\cdot\Cal S_{-1}(f)\circ\phi+\Cal S_{-1}(\phi),$
and $$\Delta(x)=-2x^{2\nu+2}\Cal S_{-1}(\tfrac{y_1}{y_2})-(\nu^2-1)x^{2\nu}.$$

\begin{remark}
The point transformations preserving the space of trace-free LDEs 	$\delta_\nu^2y=\tfrac{\Delta(x)}{4} y$
are of the form $x\mapsto\phi(x)$, $y\mapsto c\,\psi(x)^{\frac12}y$ with $\psi(x)=\tfrac{\delta_\nu\phi}{\phi^{\nu+1}}$, 
i.e. up to the multiplication by a constant $c\in\C\sminus\{0\}$ they act as transformations of the $(-\tfrac12)$-differential $y(x)\delta_\nu^{\frac{1}{2}}$, as was observed in \cite{HS}. 
\end{remark}

\begin{notation}
While we denote $\delta_\nu^2=\delta_\nu\delta_\nu,\ \delta_\nu^3=\delta_\nu\delta_\nu\delta_\nu$ etc. higher order differential operators defined by composition of $\delta_\nu$, at the same time we also write 
$\delta_\nu^{-2}=\frac{(dx)^2}{x^{2\nu+2}}$, and $\delta_\nu^{\frac{1}{2}}=x^{\frac{\nu+1}{2}}(dx)^{-\frac12}$ 
for formal powers of the differential form $\delta_\nu^{-1}=\frac{dx}{x^{\nu+1}}$. This should not cause any confusion.
\end{notation}	

\medskip

From \eqref{eq:LT-transformationp} and \eqref{eq:LT-transformationDelta} one can see that 
$$p(x)=\psi(x)\cdot\tilde p(\phi(x)) +O(x^{\nu+1}),\quad 
\Delta(x)=\psi(x)^2\cdot\tilde\Delta(\phi(x))+O(x^{2\nu+1}).$$
The natural objects to consider as invariants are the meromorphic form
$J_0^\nu p(x)\cdot\delta_\nu^{-1}$
and the meromorphic quadratic differential
$J_0^{2\nu}\Delta(x)\cdot\delta_\nu^{-2},$
where $J_0^{k}$ denotes the \emph{$k$-jet} of a germ at $x=0$,
on which the point transformations act as the usual transformation rule $x\mapsto\phi(x)$ for differentials.

\medskip

If $\Delta(0)\neq 0$, let 
$$\mu^2=\left(\res_{x=0}\sqrt{\Delta(x)}\,\delta_\nu^{-1}\right)^{\! 2},$$
be the \emph{square residue of the quadratic differential} $\Delta(x)\delta_\nu^{-2}$.
The following proposition is well known in literature, see e.g. \cite[Theorems 6.1, 6.3, 6.4]{Stre}.

\begin{theorem*}[Local normal form of the quadratic differential $\Delta(x)\delta_\nu^{-2}$]\label{prop:LT-Delta}~
\begin{enumerate}[wide=0pt, leftmargin=\parindent]
	\item If $\nu=0$ and $\mu^2=\Delta(0)\neq0$ then there exists an analytic transformation $x\mapsto\phi(x)$ such that in the new variable
	$$\Delta(x)\delta_0^{-2}=\mu^2\delta_0^{-2}.$$
	\item If $\nu>0$ then there exists an analytic transformation $x\mapsto\phi(x)$ such that in the new variable
	$$\Delta(x)\delta_\nu^{-2}=(1+\mu x^\nu)^2\delta_\nu^{-2},$$
	where $\mu\in\C$ is a root of the square residue $\mu^2$, determined up to the $\pm$ sign.
	\item If $\Delta(0)=0$, $\frac{d\Delta}{dx}(0)\neq 0$, then there exists an analytic transformation $x\mapsto\phi(x)$ such that in the new variable
	$$\Delta(x)\delta_\nu^{-2}=x\,\delta_\nu^{-2}.$$
\end{enumerate}
\end{theorem*}

In the situation we consider (regular or non-degenerate irregular singularities) it turns out that the simultaneous equivalence class of the pair of jets of forms $J_0^\nu p(x)\cdot\delta_\nu^{-1}$ and $J_0^{2\nu}\Delta(x)\cdot\delta_\nu^{-2}$
is determined by the pair $J_0^\nu p(x)\cdot\delta_\nu^{-1}$ and $J_0^{\nu}\Delta(x)\cdot\delta_\nu^{-2}$.

\begin{definition}[Formal invariants]
Let $J_0^{k}f(x)$ denote the \emph{$k$-jet} of a germ $f(x)$ at $x=0$.

We define the \emph{formal invariant} of the LDE \eqref{eq:LT-irregODE} as the simultaneous equivalence class of the pair of meromorphic forms
\begin{equation}
J_0^{\nu}p(x)\cdot\delta_\nu^{-1},\quad J_0^{\nu}\Delta(x)\cdot\delta_\nu^{-2},
\end{equation}
with respect to the action of analytic diffeomorphisms $x\mapsto\phi(x)$ and jet restriction. 

\begin{itemize}
\item If $\nu=0$, then the formal invariant is given by the pair $p(0),\ \Delta(0)=\mu^2$,
and will be identified with the pair of roots $\{\lambda_1,\lambda_2\}$ of $\lambda^2-p(0)\lambda-q(0)=0$.

\item If $\nu>0$ and  the irregular singularity is non-resonant, $\Delta(0)\neq 0$, 
let 
$$\lambda_j(x)=\lambda_j^{(0)}+\ldots+x^\nu\lambda_j^{(\nu)},\quad j=1,2,$$
be the $\nu$-jets of the roots of the characteristic polynomial $\lambda^2-p(x)\lambda-q(x)$. 
Then $J_0^{\nu}p(x)=\lambda_1(x)+\lambda_2(x)$, $J_0^{\nu}\sqrt{\Delta}(x)=\lambda_2(x)-\lambda_1(x),$
and the formal invariant can be identified with the equivalence class of the pair
$$\{\lambda_1(x)\delta_\nu^{-1},\ \lambda_2(x)\delta_\nu^{-1}\}$$
with respect to the action of analytic diffeomorphisms $x\mapsto\phi(x)$ and jet restriction. 
Moreover, one may always assume that $\Delta(x)\delta_\nu^{-2}=(1+\mu x^\nu)^2\delta_\nu^{-2}$ is in the normal form of Proposition~\ref{prop:LT-Delta} and
\begin{equation}\label{eq:LT-canonical}
\lambda_2(x)-\lambda_1(x)=1+\mu x^\nu.
\end{equation}
Such pair $\{\lambda_1(x)\delta_\nu^{-1},\ \lambda_2(x)\delta_\nu^{-1}\}$, called in \emph{canonical form}, is uniquely determined up to the action of the rotations $x\mapsto e^{\frac{l\pi i}{\nu}}x$, $l\in\Z_{2\nu}$.
In particular, if the equation is trace-free, then $\lambda_1(x)+\lambda_2(x)=0$, and the equivalence class of formal invariants is completely determined by $\nu$ and $\mu^2$. 

\item If $\nu>0$, $\Delta(0)=0$ and the resonant irregular singularity is non-degenerate, $\frac{d\Delta}{dx}(0)\neq 0$, 
then one can assume that  $\Delta(x)\delta_\nu^{-2}=x\,\delta_\nu^{-2}$, in which case 
$J_0^{\nu}p(x)\cdot\delta_\nu^{-1}$ is uniquely determined up to the action of the rotations 
$x\mapsto e^{\frac{2l\pi i}{2\nu-1}}x$, $l\in\Z_{2\nu-1 }$.
\end{itemize}
\end{definition}

\begin{definition}
A linear differential equation \eqref{eq:LT-irregODE} is called \emph{reducible} if it can be written as 
\begin{equation}\label{eq:LT-reducible}
\big(\delta_\nu-\alpha_2(x)\big)\big(\delta_\nu-\alpha_1(x)\big)y=0,
\end{equation}
with $\alpha_1(x),\ \alpha_2(x)$ analytic, $\alpha_j(x)=\sum_{k=0}^{+\infty}\alpha_j^{(k)}x^k$.
\end{definition}

\subsubsection{Regular singularities}

\begin{definition}[Projective monodromy]
	Let $y_1(x),\ y_2(x)$ be two linearly independent solutions near a point $x_0\in U^*$ of some pointed neighborhood $U^*$ of the origin, and let $f(x):=\frac{y_2(x)}{y_1(x)}$.
	For a loop $\gamma\in\pi_1(U^*,x_0)$, the analytic continuation of $f(x)$ along $\gamma$ acts on $f(x)$ as
	$$f(\gamma\cdot x)=\rho_\gamma\big(f(x)\big),\qquad\text{for some }\ \rho_\gamma\in\PGL_2(\C).$$
	The map $\rho:\gamma\mapsto\rho_\gamma$ is the \emph{projective monodromy representation} $$\rho:\pi_1(U^*,x_0)\to\PGL_2(\C)$$ of the LDE.
	This representation, which is the projectivization of the monodromy representation of the companion system, is well-defined up to conjugacy in $\PGL_2(\C)$. 
	
	Let $\gamma_0$ be a positively oriented simple loop generating $\pi_1(U^*,x_0)$. Then $\rho_{\gamma_0}$ is conjugated to either
	\begin{itemize}
		\item{(i)} $f\mapsto cf$, for some $c\in\C^*$, or
		\item{(ii)} $f\mapsto f+1$.
	\end{itemize}
	In the case \textit{(i)} we call the projective monodromy \emph{diagonalizable}, and in the case \textit{(ii)} \emph{non-diagonalizable}.
\end{definition}

\begin{lemma}\label{lemma-1}
\begin{enumerate}[wide=0pt, leftmargin=\parindent]
		\item Strongly non-resonant regular singularities (i.e. with $\mu=\lambda_2-\lambda_1\notin\Z$) have diagonalizable projective monodromy (conjugated to $f\mapsto e^{2\pi i\mu}f$).
		\item  A regular singularity has non-diagonalizable projective monodromy if and only if its formal fundamental solution contains a logarithmic term. %Such case will be also called \emph{strictly resonant}.
\end{enumerate}
\end{lemma}

\begin{theoremB1}[Analytic classification of regular singularities, $\nu=0$]\label{thm-B1}~
\begin{enumerate}[wide=0pt, leftmargin=\parindent]
\item Two LDEs \eqref{eq:LT-irregODE} with regular singularities are analytically equivalent if and only if they have the same pair of formal invariants $\{\lambda_1,\lambda_2\}$ and their projective monodromies are conjugated, i.e. either they are both diagonalizable or both non-diagonalizable.

\item A LDE \eqref{eq:LT-irregODE} with a regular singularity is always reducible. It is analytically equivalent to one of the following normal forms.
\begin{enumerate}[label=(\alph*)]
\item Diagonalizable projective monodromy ($\lambda_1\neq\lambda_2$): 
%(in particular this includes the strongly non-resonant case with $\lambda_1-\lambda_2\notin\Z$):
 	\begin{equation}\label{eq:LT-B1a} 
 	\Big(\delta_0-\lambda_2\Big)\Big(\delta_0-\lambda_1\Big)y=0,  %\delta_0^2y-(\lambda_1+\lambda_2)\delta_0y+\lambda_1\lambda_2y,
 	\end{equation}
	whose basis of solutions is $y_1(x)=x^{\lambda_1}$, $y_2(x)=x^{\lambda_2}$.
\item If $\lambda_1=\lambda_2$ (then the projective monodromy is non-diagonalizable): 
	 \begin{equation}\label{eq:LT-B1b}
 	\Big(\delta_0-\lambda_1\Big)^2y=0, %\delta_0^2y-2\lambda\delta_0y+\lambda^2y,
 	\end{equation}
	whose basis of solutions is $y_1(x)=x^{\lambda_1}$, $y_2(x)=x^{\lambda_1}\log x$.
\item If $\lambda_1-\lambda_2=k\in\Z_{>0}$ and the  projective monodromy is non-diagonalizable:      
 	\begin{equation}\label{eq:LT-B1c}
 	\Big(\delta_0-\lambda_2+k\frac{x^k}{1-x^k}\Big)\Big(\delta_0-\lambda_1\Big)y=0,
 	\end{equation}
	whose basis of solutions is $y_1(x)=x^{\lambda_1}$, $y_2(x)=x^{\lambda_2}+kx^{\lambda_1}\log x$. 
	
	Alternatively, it is also analytically equivalent to 
 	\begin{equation}\label{eq:LT-B1d}
 	\Big(\delta_0-\lambda_2+kx^k\Big)\Big(\delta_0-\lambda_1\Big)y=0.
 	\end{equation}
\end{enumerate}
\end{enumerate}
\end{theoremB1}

%Let us remark that if $\lambda_1=\lambda_2$ then the singularity is necessarily strictly resonant.

\subsubsection{Non-resonant irregular singularities}

\begin{proposition}\label{prop:LT-formalirreg}
Two non-resonant irregular LDEs \eqref{eq:LT-irregODE} are formally equivalent if and only if their pairs of formal invariants 
$\{\lambda_1(x)\delta_\nu^{-1},\ \lambda_2(x)\delta_\nu^{-1}\}$ are in the same equivalence class.
The formal transformation 
\begin{equation}\label{eq:LT-formal}
\tilde x=\hat\phi(x)=\sum_{j=1}^{+\infty}\phi^{(j)}x^j,\quad \tilde y=\hat t(x)\cdot y=\sum_{j=0}^{+\infty}t^{(j)}x^j\cdot y,\qquad
\phi^{(1)},t^{(0)}\neq0,
\end{equation}
between the two LDEs is then Borel $\nu$-summable, with singular directions $\arg(x)=\beta$ among those where 
$\Im\big(e^{-\nu\beta i}(\lambda_2^{(0)}-\lambda_1^{(0)})\big)=0$.

Assuming that their formal invariants are equal, then a formal transformation $\hat{\phi}$, $\hat t$ exists with
$\hat\phi(x)=x+tx^{\nu+1}+O(x^{\nu+2})$ which is unique (up to $y\mapsto cy$) for any $t\in\C$. 
\end{proposition}

In particular, the LDE is formally equivalent by means of a transformation \eqref{eq:LT-formal} with 
$\hat\phi(x)=x+O(x^{\nu+1})$
to the following \emph{formal normal form}
\begin{equation}\label{eq:LT-irregnormalform}
	\Big(\delta_\nu-\lambda_2(\tilde x)-\tfrac{\delta_\nu(\lambda_2(\tilde x)-\lambda_1(\tilde x))}{\lambda_2(\tilde x)-\lambda_1(\tilde x)}\Big)\Big(\delta_\nu-\lambda_1(\tilde x)\Big)\tilde y=0,
\end{equation}
i.e.
\begin{align*}
\tilde p(\tilde x)&=\lambda_1+\lambda_2+\delta_\nu\log(\lambda_2-\lambda_1), \\	
\tilde q(\tilde x)&= -\lambda_1\lambda_2+\tfrac12\delta_\nu(\lambda_1+\lambda_2)-\tfrac{\lambda_1+\lambda_2}{2}\delta_\nu\log(\lambda_2-\lambda_1),\\
\tilde\Delta(\tilde x)&=(\lambda_2-\lambda_1)^2-2\delta_\nu^2\log(\lambda_2-\lambda_1)+\left(\delta_\nu\log(\lambda_2-\lambda_1)\right)^2,
\end{align*} 
whose basis of solutions is 
$$\tilde y_j(\tilde x)=e^{\int\lambda_j(\tilde x)\delta_\nu^{-1}},\quad j=1,2.$$
The formal transformation $\hat\phi(x)=x+O(x^{\nu+1})$ is unique up to a composition with the flow of the vector field $\tfrac{1}{\lambda_2(x)-\lambda_1(x)}\delta_\nu$.

Up to an analytic change of coordinate $x\mapsto\phi(x)$, one can suppose that the pair $\{\lambda_1(x)\delta_\nu^{-1},\ \lambda_2(x)\delta_\nu^{-1}\}$
is in the canonical form 
$$\lambda_2(x)-\lambda_1(x)=1+\mu x^\nu.$$
Then the singular directions of $\hat\phi(x)$, $\hat t(x)$ are
$\beta_l=\frac{l}{\nu}\pi$, $l\in\Z$.
Let $\phi_{\Omega_l}(x)$, $t_{\Omega_l}(x)$ be the Borel sums of the formal transformation $\hat\phi$, $\hat t$, bounded and analytic on the sectors
$$\Omega_l=\left\{|\arg x-\tfrac{2l+1}{2\nu}\pi|<\tfrac{\pi}{\nu}-\eta,\ |x|<\rho_\eta \right\}, \qquad l\in\Z_{2\nu},$$
where $0<\eta<\frac{\pi}{2\nu}$ is arbitrarily small, and $\rho_\eta>0$ depends on $\eta$.
They transform the LDE to its formal normal form \eqref{eq:LT-irregnormalform}, which means that on each sector $\Omega_l$ the original LDE has a canonical basis of solutions
$$y_{j,\Omega_l}=t_{\Omega_l}(x)\cdot \tilde y_j\left(\phi_{\Omega_l}(x)\right),\qquad j=1,2.$$
We can now define projective Stokes operators of the equation, corresponding to the projectivization of the Stokes matrices of the companion system, as the operators connecting the bases on neighboring sectors  in the following way.

\begin{definition}[Projective Stokes operators]
The \emph{projective Stokes operators} are the operators $\sigma_{\beta_l}\in\PGL_2(\C)$ defined by
\begin{align*}
	f_{\Omega_{l-1}}(x)&=\sigma_{\beta_l}\left(f_{\Omega_l}(x)\right),\quad\text{for}\ \arg(x)=\tfrac{l}{\nu}\pi,\qquad l=1,\ldots,2\nu-1,\\
	f_{\Omega_{2\nu}}(e^{2\pi i}x)\cdot e^{-2\pi i\mu}&=\sigma_{\beta_0}\left(f_{\Omega_0}(x)\right),\quad\text{for}\ \arg(x)=0.
\end{align*}
where $f_{\Omega_l}(x):=\frac{y_{2,\Omega_l}(x)}{y_{1,\Omega_l}(x)}$.
They are of the form	
$$\begin{cases} \sigma_{\beta_l}:f\mapsto f+s_l &\hskip-6pt \text{if $l$ is odd, i.e. when $e^{\int (\lambda_2-\lambda_1)\delta_\nu^{-1}}$ is ``exploding'' as $x\to 0$},\\
\sigma_{\beta_l}:f\mapsto \frac{f}{1+s_l f}  &\hskip-6pt \text{if $l$ is even, , i.e. when $e^{\int (\lambda_2-\lambda_1)\delta_\nu^{-1}}$ is ``flat'' as $x\to 0$}.\end{cases}$$
Their collection $(\sigma_{\beta_0},\ldots,\sigma_{\beta_{2\nu-1}})$ is well-defined up to simultaneous conjugation by a scalar multiplication (corresponding to the non-unicity of $\hat\phi(x)$). 	
It is extended to all $l\in\Z$ by 
$$\sigma_{\beta_{l+2\nu}}(f)=e^{2\pi i\mu}\sigma_{\beta_{l}}(e^{-2\pi i\mu}f).$$
\end{definition} 

\begin{remark}
	Note that the canonical pair of solutions $y_{1,\Omega_l}(x)$, $y_{2,\Omega_l}(x)$ is up to multiplication by constants uniquely determined by their asymptotic behavior at the singular directions $\beta_l-\frac{\pi}{\nu}$ and $\beta_l+\frac{\pi}{\nu}$, one being flat at one direction the other being flat at the other direction.
\end{remark}

\begin{definition}[Symmetry group of the formal invariants]
	Let $\{\lambda_1(x)\delta_\nu^{-1},\ \lambda_2(x)\delta_\nu^{-1}\}$ be a pair of formal invariants in a canonical form.
	Let us define $G\subseteq\Z_{2\nu}$ as the subgroup of the cyclic group consisting of the elements $g\in\Z_{2\nu}$ such that
	the associated rotation $x\mapsto e^{\frac{g\pi i}{\nu}}x$ preserves the pair $\{\lambda_1(x)\delta_\nu^{-1},\ \lambda_2(x)\delta_\nu^{-1}\}$.
	Since the  pair of formal invariants in a canonical form is uniquely defined up to rotations from $\Z_{2\nu}$, which commute with $G$, the group $G$ is well-defined.
\end{definition}	

For example, if the equation is trace-free, $\lambda_1(x)+\lambda_2(x)=0$, then  either $G=\Z_{2\nu}$ if $\mu=0$, or $G=2\Z_\nu$ if $\mu\neq 0$. 

\begin{theoremB2}[Analytic classification of non-resonant irregular singularities, $\nu>0$]\label{thm-B2}~
\begin{enumerate}[wide=0pt, leftmargin=\parindent]
\item Two formally equivalent LDEs \eqref{eq:LT-irregODE} with a non-resonant irregular singularity at the origin 
and the same pair of formal invariants in canonical form $\{\lambda_1(x)\delta_\nu^{-1},\ \lambda_2(x)\delta_\nu^{-1}\}$
are analytically equivalent if and only if  their respective collections of projective Stokes operators $(\sigma_{\beta_l})_{l\in\Z}$ and $(\sigma_{\beta_l}')_{l\in\Z}$ are equivalent in the following sense:
there exist $c\in\C^*$ and $g\in G$ such that
$$\sigma_{\beta_l}'=\iota^g\circ(\tfrac{1}{c}\sigma_{\beta_{l+g}})\circ(c\iota^g)\quad\text{for all}\ l\in\Z,$$
where $G\subseteq\Z_{2\nu}$ is the symmetry group of the formal invariant and $\iota:f\mapsto \frac{1}{f}$.

\item For every pair $\{\lambda_1(x)\delta_\nu^{-1},\ \lambda_2(x)\delta_\nu^{-1}\}$ in canonical form and every collection of projective Stokes operators  $(\sigma_{\beta_0},\ldots,\sigma_{\beta_{2\nu-1}})$, there exists a LDE with a non-resonant irregular singularity of a given formal class realizing them as its analytic invariants.
\end{enumerate}
\end{theoremB2}

\begin{proposition}\label{prop:LT-reducible}~
\begin{enumerate}[wide=0pt, leftmargin=\parindent]	
\item For a non-resonant irregular singularity the following are equivalent:
\begin{enumerate}	
	\item The LDE \eqref{eq:LT-irregODE} is reducible, i.e. of the form \eqref{eq:LT-reducible} for some $\alpha_1(x)$, $\alpha_2(x)$.
	\item The LDE \eqref{eq:LT-irregODE} has a ``convergent solution'' $y(x)=e^{\int \lambda(x)\delta_\nu^{-1}}t(x)$, where $\lambda(x)$ is one of the formal invariants and $t(x)$ is a convergent power series. 
	\item The  Riccati equation	
	\begin{equation}\label{eq:LT-ricatti}
	2\delta_\nu r=r^2-\Delta(x)
	\end{equation}
	has an analytic solution $r(x)$. In this case $\alpha_1(x)=\frac{1}{2}(p(x)-r(x))$, $\alpha_2(x)=\frac{1}{2}(p(x)+r(x))$.
	\item For either all odd or all even indices $l\in\Z$ the projective Stokes operators are trivial, $\sigma_{\beta_l}=\id$.
	%\item The monodromy around 0 has for eigenvalue $e^{2\pi i\lambda^{\nu}}$, where $\lambda(x)=\lambda^{(0)}+\ldots+\lambda^{(\nu)}x^\nu$ is one of the formal invariants.
\end{enumerate}
\item For a non-resonant irregular singularity the following are equivalent:
\begin{enumerate}	
	\item The LDE \eqref{eq:LT-irregODE} is analytically equivalent to the formal normal form \eqref{eq:LT-irregnormalform}.
	\item The LDE \eqref{eq:LT-irregODE} has a pair of ``convergent solutions'' $y(x)=e^{\int \lambda_j(x)\delta_\nu^{-1}}t_j(x)$, $j=1,2$, where $\{\lambda_1(x)\delta_\nu^{-1},\lambda_2(x)\delta_\nu^{-1}\}$ are the formal invariants and $t_j(x)$ are convergent power series.
	\item The third order linear equation
	\begin{equation}\label{eq:LT-symmetricpower}
	\delta_\nu^3 h-\Delta(x)\delta_\nu h-\tfrac{1}{2}\big(\delta_\nu\Delta(x)\big)h=0.
	\end{equation} 	
	has an analytic solution $h(x)$.
	\item The  Riccati equation	\eqref{eq:LT-ricatti} has two different analytic solutions $r_1(x)$, $r_2(x)$.
	In this case $h(x)=\frac{1}{r_2(x)-r_1(x)}$ is an analytic solution to \eqref{eq:LT-symmetricpower}.
	\item All the projective Stokes operators are trivial, $\sigma_{\beta_l}=\id$ for all $l\in\Z$.
\end{enumerate}
\end{enumerate}
\end{proposition}

The differential operator of the left-hand side of \eqref{eq:LT-symmetricpower} is known as the \emph{second symmetric power} of the operator $\delta_\nu^2-\frac{\Delta(x)}{4}$ \cite[\S 2.3]{SP}, \cite[\S 3.1]{OP}.
If the LDE is reducible, and $r(x)$ an analytic solution to \eqref{eq:LT-ricatti}, then
the equation \eqref{eq:LT-symmetricpower} can be factorized as
\begin{equation}\label{eq:LT-symmetricfactorized}
\big(\delta_\nu-r(x)\big)\delta_\nu\big(\delta_\nu+r(x)\big) h=0.
\end{equation}
Since in this case $r(0)\neq 0$, the formal power series solutions to \eqref{eq:LT-symmetricfactorized} are also solutions to $\delta_\nu\big(\delta_\nu+r(x)\big) h=0$.

\begin{remark}
	If $y_1(x)$, $y_2(x)$ are two linearly independent solutions to the LDE, $f(x)=\frac{y_2(x)}{y_1(x)}$,
	then $r=\frac{\delta_\nu^2 f}{\delta f}=p-2\frac{\delta_\nu y_1}{y_1}$ is a solution to \eqref{eq:LT-ricatti},
	and $h=\frac{f}{\delta_\nu f}$ is a solution to \eqref{eq:LT-symmetricpower}.
\end{remark}

\begin{theoremD}[Analytic normal forms when $\nu=1$]\label{thm-D}~
\begin{enumerate}[wide=0pt, leftmargin=\parindent]
	\item An irreducible LDE with a non-resonant irregular singularity at the origin of Poincar\'e rank $\nu=1$ 
	is analytically equivalent to an LDE of the form
	$$\delta_1^2y-\big(p^{(0)}+xp^{(1)}\big)\delta_1y-\big(q^{(0)}+xq^{(1)}+x^2q^{(2)}\big)y,$$
	with $1=\Delta^{(0)}=\big(p^{(0)}\big)^2+4q^{(0)}$ and $\mu=\frac{\Delta^{(1)}}{2}=p^{(0)}p^{(1)}+2q^{(1)}\in\C$. 
		
	Two such equations are analytically equivalent if and only if
	%$$p^{(0)}+xp^{(1)}=\pm\tilde p^{(0)}+x\tilde p^{(1)},\qquad q^{(0)}+xq^{(1)}=\tilde q^{(0)}\pm x\tilde q^{(1)},$$
	$$\mu=\pm\tilde{\mu},\quad\text{and}\quad \cos\pi\sqrt{\Delta^{(2)}+1}=\cos\pi\sqrt{\tilde\Delta^{(2)}+1},$$
	where $\Delta^{(2)}=(p^{(1)})^2+4q^{(2)}-2p^{(1)}$.
		
	\item A reducible LDE \eqref{eq:LT-reducible} with a non-resonant irregular singularity at the origin of Poincar\'e rank $\nu=1$ 
	with $\mu=\alpha_2^{(1)}-\alpha_1^{(1)}\notin\Z_{\leq 0}$ (hence with  diagonalizable monodromy)
	is analytically equivalent to
	either
		\begin{equation}\label{eq:LT-reduciblenf1}
		\Big(\delta_1-\lambda_2(x)\Big)\Big(\delta_1-\lambda_1(x)\Big)y=0, 
		\end{equation}
	with $\lambda_2(x)-\lambda_1(x)=1+\mu x$, or to \eqref{eq:LT-irregnormalform}.
		
	\item A reducible LDE \eqref{eq:LT-reducible} with a non-resonant irregular singularity at the origin of Poincar\'e rank $\nu=1$ with
	$\mu=\alpha_2^{(1)}-\alpha_1^{(1)}\in\Z_{\leq 0}$ is analytically equivalent to  
	either \eqref{eq:LT-reduciblenf1}, which in this case is analytically equivalent to \eqref{eq:LT-irregnormalform}, 
	if the monodromy is diagonalizable (scalar), 
	or to
		\begin{equation}\label{eq:LT-reduciblenf2}
		\Big(\delta_1-\lambda_2(x)+x^{2}\Big)\Big(\delta_1-\lambda_1(x)\Big)y=0,
		\end{equation}
		with $\lambda_2(x)-\lambda_1(x)=1+\mu x$, if the monodromy is non-diagonalizable.
\end{enumerate}
\end{theoremD}

\subsubsection{Non-degenerate resonant irregular singularities}

\begin{proposition}\label{prop:LT-formalres}
	Two non-degenerate resonant irregular LDEs \eqref{eq:LT-irregODE} are formally equivalent if and only if their pairs of formal invariants 
	$J_0^{\nu}p(x)\cdot\delta_\nu^{-1}$, $J_0^{\nu}\Delta(x)\cdot\delta_\nu^{-2}$ are from the same equivalence class.
	The formal transformation 
	\begin{equation*}%\label{eq:LT-formal1}
	\tilde x=\hat\phi(x)=\sum_{j=1}^{+\infty}\phi^{(j)}x^j,\quad \tilde y=\hat t(x)\cdot y=\sum_{j=0}^{+\infty}t^{(j)}x^j\cdot y,\qquad 
	\phi^{(1)},t^{(0)}\neq0,
	\end{equation*}
	between the two LDEs is then Borel $(\nu-\tfrac12)$-summable, with singular directions $\arg(x)=\beta$ among those where 
	$\Im\big(e^{(1-2\nu)\beta i}\frac{d\Delta}{dx}(0)\big)=0$.
	
	Assuming that their formal invariants are equal, then a unique  formal transformation $\hat{\phi}$, $\hat t$ exists with
	$\hat\phi(x)=x+O(x^{\nu+1})$ and $\hat t(0)=1$. 
\end{proposition}

Suppose $J_0^\nu\Delta(x)=x$ and $J_0^\nu p(x)=P(x)$, $P(0)=0$.  
Consider the formal normal form LDE \eqref{eq:LT-irregnormalform} with 
$\lambda_1(x)+\lambda_2(x)=P(x)-\tfrac12 x^\nu$, $\lambda_2(x)-\lambda_1(x)= x^{\frac12}$,
that is the LDE 
\begin{equation}\label{eq:LT-nf}
\delta_\nu^2\tilde y-\tilde p(x)\delta_\nu\tilde y-\tilde q(x)\tilde y=0,
\end{equation}
with
\begin{align*}
\tilde p(x)&=P(x), \\	
\tilde q(x)&=\tfrac14\left[x-P(x)^2+x^\nu P(x)-(\nu+\tfrac14)x^{2\nu}+2\delta_\nu P(x)\right],\\ 
\tilde\Delta(x)&=x+x^\nu P(x)-(\nu+\tfrac14)x^{2\nu},
\end{align*} 
whose basis of solutions is 
$\tilde y_j(x)=e^{\int\lambda_j(x)\delta_\nu^{-1}}$, $j=1,2.$

By the Proposition~\ref{prop:LT-formalres}, the LDE is equivalent to \eqref{eq:LT-nf} by a formal transformation $\hat\phi(x)$, $\hat t(x)$, 
Borel $(\nu-\tfrac12)$-summable except in the singular directions
$$\beta_l=\tfrac{2l}{2\nu-1}\pi,\qquad l\in\Z.$$ 
Hence the LDE has a canonical basis of solutions
$$y_{j,\Omega_l}=t_{\Omega_l}(x)\cdot \tilde y_j\left(\phi_{\Omega_l}(x)\right),\qquad j=1,2,$$
where
$\phi_{\Omega_l}(x)$, $t_{\Omega_l}(x)$ are the Borel sums, bounded and analytic on the sectors
$$\Omega_l=\left\{|\arg x-\tfrac{2l+1}{2\nu-1}\pi|<\tfrac{2\pi}{2\nu-1}-\eta,\ |x|<\rho_\eta \right\}, \qquad l\in\Z_{2\nu-1},$$
where $0<\eta<\frac{\pi}{2\nu-1}$ is arbitrarily small, and $\rho_\eta>0$ depends on $\eta$.

\medskip
The \emph{projective Stokes operators} $\sigma_{\beta_l}\in\PGL_2(\C)$ are now defined as before
\begin{align*}
	f_{\Omega_{l-1}}(x)&=\sigma_{\beta_l}\left(f_{\Omega_l}(x)\right),\quad\text{for}\ \arg(x)=\tfrac{l}{\nu}\pi,\qquad l=1,\ldots,2\nu-1,\\
	f_{\Omega_{2\nu}}(e^{2\pi i}x)^{-1}&=\sigma_{\beta_0}\left(f_{\Omega_0}(x)\right),\quad\text{for}\ \arg(x)=0.
\end{align*}
where $f_{\Omega_l}(x):=\frac{y_{2,\Omega_l}(x)}{y_{1,\Omega_l}(x)}$.
Their definition is extended to all $l\in\Z$ by 
$$\sigma_{\beta_{l+2\nu}}=\iota\circ\sigma_{\beta_{l}}\circ\iota,\quad\text{where }\ \iota:f\mapsto\tfrac{1}{f}.$$
They are of the form	
$$\begin{cases} \sigma_{\beta_l}:f\mapsto f+s_l &\hskip-6pt \text{when $e^{\int (\lambda_2-\lambda_1)\delta_\nu^{-1}}$ is ``exploding'' as $x\to0$},\\
\sigma_{\beta_l}:f\mapsto \frac{f}{1+s_l f}  &\hskip-6pt \text{when $e^{\int (\lambda_2-\lambda_1)\delta_\nu^{-1}}$ is ``flat'' as $x\to 0$}.\end{cases}$$

\begin{definition}[Symmetry group of the formal invariants]
	Suppose $J_0^{\nu}\Delta(x)=x$, $J_0^{\nu}p(x)=P(x)$.
	Let us define $G\subseteq\Z_{2\nu-1}$ as the subgroup of the cyclic group consisting of the elements $g\in\Z_{2\nu-1}$ such that
	the associated rotation $x\mapsto e^{\frac{g2\pi i}{2\nu-1}}x$ preserves the differential form $P(x)\delta_\nu^{-1}$.
	Since $P(x)\delta_\nu^{-1}$ is uniquely defined up to rotations from $\Z_{2\nu-1}$, which commute with $G$, the group $G$ is well-defined.
\end{definition}

 \begin{theoremB3}[Analytic classification of non-degenerate resonant irregular singularities, $\nu>0$]\label{thm-B3}~
	\begin{enumerate}[wide=0pt, leftmargin=\parindent]
		\item Two formally equivalent LDEs \eqref{eq:LT-irregODE} with a non-degenerate resonant irregular singularity at the origin and the same pair of formal invariants in canonical form $x\delta_\nu^{-2}$, 
		$P(x)\delta_\nu^{-1}$
		are analytically equivalent if and only if  their respective collections of projective Stokes operators $(\sigma_{\beta_l})_{l\in\Z}$ and $(\sigma_{\beta_l}')_{l\in\Z}$ are equivalent in the following sense:
		there exist $g\in G$ such that
		$$\sigma_{\beta_l}'=\iota^g\circ\sigma_{\beta_{l+g}}\circ\iota^g\quad\text{for all}\ l\in\Z,$$
		where $G\subseteq\Z_{2\nu-1}$ is the symmetry group of the formal invariant and $\iota:f\mapsto \frac{1}{f}$.
		
		\item For every pair of formal invariants $x\delta_\nu^{-2}$, $P(x)\delta_\nu^{-1}$ in canonical form and every collection of projective Stokes operators, there exists a LDE with a non-degenerate resonant irregular singularity of given formal class realizing them as its analytic invariants.
	\end{enumerate}
\end{theoremB3}

\begin{theoremD}[Analytic normal forms when $\nu=1$]\label{thm-D2}~
LDE with a non-degenerate resonant irregular singularity at the origin of Poincar\'e rank $\nu=1$ is analytically equivalent to an LDE of the form
$$\delta_1^2y-p^{(1)}x\delta_1y-\big(\tfrac14x+x^2q^{(2)}\big)y,$$
with $\Delta(x)=x+x^2\Delta^{(2)}$, with $\Delta^{(2)}=(p^{(1)})^2+4q^{(2)}-2p^{(1)}$. 
Two such equations are analytically equivalent if and only if
$$p^{(1)}=\tilde p^{(1)},\quad\text{and}\quad \cos\pi\sqrt{1+\Delta^{(2)}}=\cos\pi\sqrt{1+\tilde\Delta^{(2)}}.$$
\end{theoremD}

\subsection{Meromorphic classification}\label{sec:1.4} 
If the transformation \eqref{eq:LT-ODEpoint} is meromorphic with $t(x)=x^m u(x)$, $u(0)\neq 0$, then from \eqref{eq:LT-transformationp}--\eqref{eq:LT-transformationDelta} one can see that 
$$p(x)=\psi(x)\cdot\tilde p(\phi(x))-2mx^\nu+O(x^{2\nu+1}),\quad 
\Delta(x)=\psi(x)^2\cdot\tilde\Delta(\phi(x)) +O(x^{\nu+1}).$$
In the regular or non-resonant irregular case this means that $\lambda_{1,2}(x)=\psi(x)\tilde\lambda_{1,2}(\phi(x))-mx^\nu+O(x^{\nu+1})$,
i.e. that equivalence class of the pair of formal invariants $\{\lambda_1(x)\delta_\nu^{-1},\lambda_2(x)\delta_\nu^{-1}\}$ is shifted by $-mx^\nu\delta_\nu^{-1}$.

\begin{theoremE}[Non-resonant irregular case, $\nu>0$]\label{thm-E}
	Two LDEs \eqref{eq:LT-irregODE} with either regular or non-degenerate irregular singularities are meromorphically equivalent if and only if 
	they have the same equivalence class of formal invariants 
	$J_0^\nu\Delta(x)\cdot\delta_\nu^{-2},\ J_0^\nu p(x)\cdot\delta_\nu^{-1}$ 
	up to a shift $J_0^\nu p(x)\cdot\delta_\nu^{-1}\mapsto J_0^\nu p(x)\cdot\delta_\nu^{-1}+mx^\nu\delta_\nu^{-1}$, $m\in\Z$,
	and 
	\begin{enumerate}
		\item if regular: their monodromies are conjugated (i.e. they are both either diagonalizable or non-diagonalizable),
		\item if non-resonant irregular: their collections of Stokes operators are equivalent in the sense of Theorem~\ref{thm-B2},
		\item if non-degenerate resonant irregular: their collections of Stokes operators are equivalent in the sense of Theorem~\ref{thm-B3}.
	\end{enumerate}
\end{theoremE}

Let us remark that, unlike for systems, conjugation of monodromies of regular singularities alone does not suffice to produce meromorphic equivalence.

\subsection{Lie symmetries}\label{sec:1.5} 
For non-resonant singularities of differential systems there is a canonical diagonal formal normal form
\begin{equation}\label{eq:diagonalsystem}
\delta_\nu u=\left(\begin{smallmatrix} \lambda_1(x) & 0 \\[3pt] 0 & \lambda_2(x)\end{smallmatrix}\right)u,
\end{equation}
which is integrable in terms of elementary functions with fundamental solution matrix  
$\left(\begin{smallmatrix} e^{\int\lambda_1(x)\delta_\nu^{-1}} & 0 \\[3pt] 0 & e^{\int\lambda_2(x)\delta_\nu^{-1}}\end{smallmatrix}\right)$.
Correspondingly, the analytic class of this normal form system is the one that has the largest possible Lie algebra of analytic infinitesimal symmetries (see \cite{BMW}) of all the systems within the formal class.
It turns out that the same holds also for non-resonant irregular LDEs \eqref{eq:LT-irregODE}: they are analytically equivalent to their formal normal form \eqref{eq:LT-irregnormalform} if and only their Lie algebra of linear analytic point symmetries is the largest possible (Theorem~\ref{thm-F} below).

Let us recall that an \emph{infinitesimal linear symmetry} of a LDE \eqref{eq:LT-ODE} is a vector field 
$$Y=g(x)\tfrac{\partial}{\partial x}+f(x)y\tfrac{\partial}{\partial y},$$
whose second jet prolongation $\pr^{(2)} Y$ leaves the surface
$y_{xx}+a_1(x)y_x+a_0(x)y=0$ %$y_{xx}-[\frac{p(x)}{x^{\nu+1}}-\frac{\nu+1}{x}]y_x-\frac{q(x)}{x^{2\nu+2}}$ 
invariant.
This is equivalent \cite[p.350]{Dic} to ask that 
\begin{equation}\label{eq:LT-Lie}
 [X,\pr^{(1)}Y]=\alpha(x)X, \qquad\text{for some function } \alpha(x),
\end{equation}
where 
\begin{equation}\label{eq:LT-LieX}
X=\tfrac{\partial}{\partial x} + y_x\tfrac{\partial}{\partial y} + \big(a_1(x)y_x+a_0(x)y\big)\tfrac{\partial}{\partial y_x},
%X=\tfrac{\partial}{\partial x} + y_x\tfrac{\partial}{\partial y} + \big([\tfrac{p(x)}{x^{\nu+1}}-\tfrac{\nu+1}{x}]y_x+\tfrac{q(x)}{x^{2\nu+2}}y\big)\tfrac{\partial}{\partial y_x},
\end{equation}
and $\pr^{(1)}Y$ is the first jet prolongation of $Y$:
\begin{equation}\label{eq:LT-LieY}
\pr^{(1)}Y=g(x)\tfrac{\partial}{\partial x}+f(x)y\tfrac{\partial}{\partial y}+\big(\tfrac{df}{dx}(x)y+f(x)y_x-\tfrac{dg}{dx}(x)y_x\big)\tfrac{\partial}{\partial y_x}.
\end{equation}

\begin{theoremF}\label{thm-F}
The infinitesimal linear symmetries of a LDE \eqref{eq:LT-irregODE} are of the form
$$Y=h(x)\delta_\nu+\tfrac{1}{2}\big(c+\delta_\nu h(x)+p(x)h(x)\big)y\tfrac{\partial}{\partial y}, $$
where $c\in\C$, and $h(x)$ is a solution of \eqref{eq:LT-symmetricpower}.
%\begin{equation}\label{eq:LT-symmetricpower}
%	\delta_\nu^3 h-\Delta(x)\delta_\nu h-\tfrac{1}{2}\big(\delta_\nu\Delta(x)\big)h=0.
%\end{equation}
The Lie algebra of analytic infinitesimal linear symmetries of
\begin{enumerate}[wide=0pt, leftmargin=\parindent]
	\item a regular singularity 
	\begin{enumerate}
	\item strongly non-resonant with $\lambda_1-\lambda_2\notin\Z\sminus\{0\}$ in the normal form \eqref{eq:LT-B1a} is generated by
	$$y\tfrac{\partial}{\partial y}, \quad  \delta_0, $$
	\item resonant with $\lambda_1-\lambda_2=k\in\Z_{>0}$ and trivial projective monodromy, in the normal form \eqref{eq:LT-B1a}, is generated by
	$$y\tfrac{\partial}{\partial y}, \quad  \delta_0, \qquad x^k\Big(\delta_0+\lambda_1y\tfrac{\partial}{\partial y}\Big), %\quad 	x^{-k}\Big(\delta_0+\lambda_2y\tfrac{\partial}{\partial y}\Big),
	$$
	\item resonant with $\lambda_1-\lambda_2=k\in\Z_{>0}$ and non-diagonalizable projective monodromy, in the normal form \eqref{eq:LT-B1c},
	$$y\tfrac{\partial}{\partial y}, \quad \tfrac{x^k}{1-x^k}\Big(\delta_0+\lambda_1 y\tfrac{\partial}{\partial y}\Big),$$
	\end{enumerate}
	\item a non-resonant irregular singularity of Poincar\'e rank $\nu>0$
	\begin{enumerate}
	\item in the normal form \eqref{eq:LT-irregnormalform} is generated by
	$$y\tfrac{\partial}{\partial y}, \quad \tfrac{1}{\lambda_2(x)-\lambda_1(x)}\Big(\delta_\nu+\tfrac{\lambda_1(x)+\lambda_2(x)}{2}y\tfrac{\partial}{\partial   y}\Big),$$
	\item not analytically equivalent to \eqref{eq:LT-irregnormalform} is generated only by
	$$y\tfrac{\partial}{\partial y}.$$
	\end{enumerate}
\item a non-degenerate resonant irregular singularity of Poincar\'e rank $\nu>0$ is generated only by
	$$y\tfrac{\partial}{\partial y}.$$
\end{enumerate}
\end{theoremF}

Let us remark that for non-singular LDEs  the Lie algebra of analytic linear infinitesimal symmetries is of maximal dimension 4, namely for $\frac{d^2y}{dx^2}=0$ it is generated by
$$y\tfrac{\partial}{\partial y}, \quad \tfrac{\partial}{\partial x}, \quad x\tfrac{\partial}{\partial x}, \quad x\big(x\tfrac{\partial}{\partial x}+y\tfrac{\partial}{\partial y}\big),$$
while the Lie algebra of \emph{all} analytic infinitesimal symmetries, i.e. infinitesimal point symmetries of the form $Y=G(x,y)\partial_x+F(x,y)\partial_y$, is of dimension 8 (it is well known to be isomorphic to $\mathfrak{sl}_3(\C)$). %\cite{Lie}.

\section{Proofs}
We will review and adapt to our needs some basics of the theory of singularities of linear differential systems which can be found in some form in most standard references, e.g. in \cite{BaVa,Ba,BJL12,IlYa,Mal1,SP,Sib}.
In the case of regular singularities the analytic classification agrees with a formal one, and in the case of non-resonant irregular singularities 
the analytic modulus consists of a set of formal invariants and of a conjugacy equivalence class of a collection of Stokes matrices.

\begin{proof}[\textbf{Proof of Theorem~\ref{thm-A}}]
	Follows from Theorems~\ref{thm-B1}, \ref{thm-B2} and~\ref{thm-B3}.
\end{proof}

\paragraph{Regular singular points.}
If $\nu=0$, the singular point is of Fuchsian kind and, according to the general theory \cite[Theorem 16.16]{IlYa}, the companion system \eqref{eq:LT-irregsyst} is analytically equivalent by a gauge transformation $v=T(x)\tilde v$ to a normal form
\begin{equation}\label{eq:LT-syst}
 \delta_0 \tilde v=\begin{pmatrix} \lambda_1 & \epsilon x^{\lambda_1-\lambda_2}\\ 0& \lambda_2 \end{pmatrix}\tilde v,
\end{equation}
where $\epsilon\in\{0,1\}$ and $\epsilon\neq 0$ only if $\lambda_1-\lambda_2\in\Z_{\geq 0}$.
Therefore the system \eqref{eq:LT-irregsyst}  possesses a fundamental solution matrix 
$V(x)=T(x)\left(\begin{smallmatrix} x^{\lambda_1} & \epsilon x^{\lambda_1}\log x \\[3pt] 0& x^{\lambda_2} \end{smallmatrix}\right)$, 
where $T(x)=\big(T_{ij}(x)\big)$ is analytic and 
$$T(0)=\begin{cases}
\left(\begin{smallmatrix} 1 & 1 \\[3pt] \lambda_1 & \lambda_2 \end{smallmatrix}\right) & \text{if } \ \lambda_1\neq\lambda_2,\\[6pt]
\left(\begin{smallmatrix} 1 & 0 \\[3pt] \lambda_1 & 1 \end{smallmatrix}\right) & \text{if}\ \lambda_1=\lambda_2\ \text{in which case } \epsilon=1.
\end{cases}$$
%($T(0)$ conjugates the leading matrix $\left(\begin{smallmatrix} 0 & 1 \\[3pt] -\lambda^2 & 2\lambda \end{smallmatrix}\right)$ of the companion system to $\left(\begin{smallmatrix} \lambda & \epsilon \\[3pt] 0 & \lambda \end{smallmatrix}\right)$).
The complete analytic invariant of the system is given by the pair $\{\lambda_1,\lambda_2\}$ and by $\epsilon\in\{0,1\}$.
The monodromy matrix $M$ of this fundamental solution, $V(e^{2\pi i}x)=V(x)M$, is then given by
$M=\left(\begin{smallmatrix} e^{2\pi i\lambda_1} & \epsilon\, 2\pi i e^{2\pi i\lambda_1} \\[3pt] 0 & e^{2\pi i\lambda_2} \end{smallmatrix}\right).$

The LDE has therefore a solution basis
\begin{equation}\label{eq:LT-regsol}
y_1(x)=T_{11}(x)x^{\lambda_1},\qquad y_2(x)=T_{12}(x)x^{\lambda_2}+\epsilon T_{11}(x)x^{\lambda_1}\log x.
\end{equation}

\begin{proof}[\textbf{Proof of Lemma~\ref{lemma-1}}]
The weak non-resonance condition is equivalent to $\epsilon=0$ as well as to the diagonalizability of the monodromy $M$.	
\end{proof}

\begin{proof}[\textbf{Proof of Theorem~\ref{thm-B1}}]
The statement \textit{1} is a corollary of \textit{2}.	

\smallskip\noindent
\textit{2(a) Weakly non-resonant regular singularity:}
By the above considerations, the LDE \eqref{eq:LT-irregODE} has a solution basis 
\eqref{eq:LT-regsol} with $\epsilon=0$, where $T_{1j}(x)$ is an analytic germ with $T_{1j}(0)=1$.
We are looking for an analytic transformation $\tilde x=\phi(x),\ \tilde y=t(x)y$ \eqref{eq:LT-ODEpoint}, 
such that
$$T_{1j}(x)x^{\lambda_j}=t(x)\cdot\phi(x)^{\lambda_j},\ j=1,2.$$
Writing $\phi(x)=x(1+g(x))$, $g(0)=0$,
then $g(x)$ is a solution to
\begin{align*}
\log\left(\tfrac{T_{12}}{T_{11}}\right)\!\!(x)=(\lambda_2-\lambda_1)\log(1+g(x)),
\end{align*}
where the right-hand side is an analytic function of $x$ and $g$ whose derivative with respect to $g$ at $(x,g)=0$ is $\lambda_2-\lambda_1\neq 0$, so it has by the implicit function theorem a unique analytic solution $g(x)$ with $g(0)=0$.
Then also $t(x)=T_{11}(x)(1+g(x))^{-\lambda_1}$ is an analytic germ,  $t(0)=1$.

\medskip\noindent
\textit{2(b) Regular singularity with $\lambda_1=\lambda_2=:\lambda$:}
The LDE \eqref{eq:LT-irregODE} has a solution basis 
\eqref{eq:LT-regsol} with $\epsilon=1$, where $T_{11}(0)=1$, $T_{12}(0)=0$.
The transformation equation we want to solve is
\begin{equation*}
T_{11}(x)x^{\lambda}=t(x)\cdot\phi(x)^{\lambda},\qquad T_{12}(x)x^{\lambda}+T_{11}(x)x^{\lambda}\log x=t(x)\cdot\phi(x)^{\lambda}\log\phi(x).
\end{equation*}
Writing $\phi(x)=x(1+g(x))$, $g(0)=0$,
then 
\begin{align*}
\tfrac{T_{12}}{T_{11}}(x)=\log(1+g(x)), \quad\text{hence } \ g(x)=e^{\frac{T_{12}}{T_{11}}(x)}-1,
\end{align*}
and $t(x)=T_{11}(x)(1+g(x))^{-\lambda}$, $t(0)=1$. 

\medskip\noindent
\textit{2(c) Strictly resonant regular singularity with $\lambda_1-\lambda_2=k\in\Z_{>0}$:}
The LDE \eqref{eq:LT-irregODE} has a solution basis 
\eqref{eq:LT-regsol} with $\epsilon=1$, where $T_{11}(0)=T_{12}(0)=1$.
The transformation equation we want to solve is
\begin{align*}
T_{11}(x)x^{\lambda_1}&=t(x)\cdot\phi(x)^{\lambda_1},\\
k\left(T_{12}(x)x^{\lambda_2}+T_{11}(x)x^{\lambda_1}\log x\right)&=t(x)\cdot\left(\phi(x)^{\lambda_2}+k\phi(x)^{\lambda_1}\log\phi(x)\right).\\
\end{align*}
Writing $\phi(x)=cx(1+g(x))$, $g(0)=0$,
let $c=k^{-\frac{1}{k}}$, then $g(x)$ is a solution to
\begin{align*}
\tfrac{T_{12}}{T_{11}}(x)=(1+g(x))^{-k}+x^k\log(1+g(x))-\tfrac{x^k}{k}\log k.
\end{align*}
The derivative of the right side with respect to $g$ at $(x,g)=0$ is $\lambda_2-\lambda_1=-k$, therefore the equation
has a unique analytic solution $g(x)$ with $g(0)=0$.
Then $t(x)=T_{11}(x)c^{-\lambda_1}(1+g(x))^{-\lambda_1}$, $t(0)=c^{-\lambda_1}$. 
%\end{itemize}
\end{proof}

\paragraph{Non-resonant irregular singular points.}
Let $\lambda_j(x)=\lambda_j^{(0)}+\ldots+\lambda_j^{(\nu)}x^\nu$, $j=1,2$, be modulo $x^{\nu+1}$ the roots of the characteristic polynomial $\lambda^2-p(x)\lambda-q(x)=0$.
If $\nu>0$ and $\lambda_1(0)\neq\lambda_2(0)$, then the singular point is non-resonant irregular and, according to the general theory \cite[\S 20]{IlYa}, the companion system \eqref{eq:LT-irregsyst} possesses a formal fundamental solution matrix 
$\hat V(x)=\hat T(x)e^{\int\left(\begin{smallmatrix} \lambda_1 & 0\\ 0& \lambda_2 \end{smallmatrix}\right)\delta_\nu^{-1}}$, where $\hat T(x)=\left(\hat T_{ij}(x)\right)$ is a formal power series, $\hat T(0)=\left(\begin{smallmatrix} 1 & 1 \\[3pt] \lambda_1^{(0)} & \lambda_2^{(0)} \end{smallmatrix}\right)$.
Correspondingly, the LDE has a formal solution basis
\begin{equation}\label{eq:LT-irregsol}
\hat y_1(x)=\hat T_{11}(x)e^{\int\lambda_1(x)\delta_\nu^{-1}},\qquad \hat y_2(x)=\hat T_{12}(x)e^{\int\lambda_2(x)\delta_\nu^{-1}}.
\end{equation}
A complete \emph{formal invariant} of the system \eqref{eq:LT-irregsyst} with respect to formal gauge transformations fixing $x$ \eqref{eq:LT-systemgauge} is formed by the pair of meromorphic 1-forms $\{\lambda_1(x)\delta_\nu^{-1},\lambda_2(x)\delta_\nu^{-1}\}$.
If one allows also for transformations $x\mapsto\phi(x)$, then it is always possible to transform analytically the pair to a canonical form where $\big(\lambda_2(x)-\lambda_1(x)\big)\delta_\nu^{-1}=(1+\mu\tilde x^{\nu})\tilde\delta_\nu^{-1}$, where $\mu$ is  well-defined up to the $\pm$ sign.
%Any further transformation that preserves the polar part of this normalized form up to change of sign is written as $\phi(x)=e^{\frac{\pi il}{\nu}}x+O(x^{\nu+1})$ for some $l\in\Z_{2\nu}$. 

A \emph{Stokes direction} (also known as \emph{separating}) $\alpha\in\R$ is defined by $\Re\big(e^{-\nu\alpha}(\lambda_2^{(0)}-\lambda_1^{(0)})\big)=0$,
and an \emph{anti-Stokes direction} (also known as \emph{singular}) $\beta\in\R$ is defined by $\Im\big(e^{-\nu\beta}(\lambda_2^{(0)}-\lambda_1^{(0)})\big)=0$.
After the normalization $\lambda_2^{(0)}-\lambda_1^{(0)}=1$, this means $\alpha\in\frac{\pi}{\nu}\Z$, $\beta\in\frac{\pi}{2\nu}+\frac{\pi}{\nu}\Z$.
Let $\{\alpha_l\}_{l\in\Z}$, resp. $\{\beta_l\}_{l\in\Z}$, be all the Stokes directions, resp. anti-Stokes directions in their order, $\alpha_{l+2\nu}=\alpha_l+2\pi$. 
By a classical theorem of Horn, Trjitzinsky, Hukuhara, Turittin, and others \cite[Theorem 20.16]{IlYa}, the formal power series $\hat T(x)$ is asymptotic to a unique bounded analytic matrix-valued function $T_{\alpha_l}(x)$ on every open sector $\Omega_l$ covering exactly one Stokes direction $\arg(x)=\alpha_l$.
Let $\Omega_{2\nu}=\Omega_0,\ldots,\Omega_{2\nu-1}$ be a cyclic covering by such sectors of a pointed neighborhood of the origin, 
such that the intersection of two neighboring sectors covers exactly one anti-Stokes direction (when considered on the Riemann surface of $\log x$),
let $T_{\alpha_{2\nu}}=T_{\alpha_0},\ \ldots,\ T_{\alpha_{2\nu-1}}$ be the associated sectoral transformations,
and let $V_{\alpha_l}(x)=T_{\alpha_l}(x)e^{\int\left(\begin{smallmatrix} \lambda_1(x) & 0\\ 0& \lambda_2(x) \end{smallmatrix}\right)\delta_\nu^{-1}}$ be the sectoral fundamental matrix solutions, 
$V_{\alpha_l+2\pi}(x)=V_{\alpha_l}(x)e^{2\pi i\left(\begin{smallmatrix} \lambda_1^{(\nu)} & 0\\ 0& \lambda_2^{(\nu)} \end{smallmatrix}\right)}$.
For each anti-Stokes direction $\beta$ the \emph{Stokes matrix} $St_\beta$ is defined by
$$V_{\beta-\frac{\pi}{2\nu}}=V_{\beta+\frac{\pi}{2\nu}}St_{\beta}.$$
The collection of the Stokes matrices $\{St_{\beta_0},\ldots St_{\beta_{2\nu-1}}\}$ modulo simultaneous conjugation by diagonal matrices is a complete analytic invariant of the system \eqref{eq:LT-irregsyst} with given formal invariants (sometimes called Malgrange--Sibuya modulus).

\begin{proof}[\textbf{Proof of Proposition~\ref{prop:LT-formalirreg}}]
By the above considerations, the LDE \eqref{eq:LT-irregODE} has a formal solution basis 
\eqref{eq:LT-irregsol}, where $\hat T_{1j}(x)$ is formal power series that is Borel $\nu$-summable, $\hat T_{1j}(0)\neq 0$.
After an eventual analytic change of the $x$-variable, we can suppose that the formal invariants $\{\lambda_1(x)\delta_\nu^{-1},\lambda_2(x)\delta_\nu^{-1}\}$ are in a canonical form with $\lambda_2(x)-\lambda_1(x)=1+\mu x^\nu$.
We are looking for a formal transformation 
\begin{equation*}%\label{eq:LT-ODEpoint}
	\tilde x=\hat\phi(x)=x(1+x^{\nu}\hat g(x)),\qquad \tilde y=\hat t(x)y,
\end{equation*}
such that
$$\hat T_{1j}(x)e^{\int \lambda_j(x)\delta_\nu^{-1}}=\hat t(x)\cdot\big(e^{\int \lambda_j\delta_\nu^{-1}}\big)\circ\hat\phi(x),\ j=1,2.$$
Therefore $\hat g(x)$ is a solution to
\begin{align*}
	\log\left(\tfrac{\hat T_{12}}{\hat T_{11}}\right)\!\!(x)=\frac{1}{\nu x^\nu}-\frac{1}{\nu x^\nu(1+x^\nu\hat g)^\nu}+
	\mu(1+x^\nu\hat g),
\end{align*}
where the right-hand side is an analytic function of $x$ and $g$ whose derivative with respect to $\hat g$ at $x=0$ is $\lambda_2^{(0)}-\lambda_1^{(0)}=1$, so  by the formal implicit function theorem it has a unique formal solution $\hat g(x)$ with $\hat g(0)=\log\big(\frac{\hat T_{12}(0)}{\hat T_{11}(0)}\big)$.
Since $\hat\phi(x)=x+O(x^{\nu+1})$ then also $\hat t(x)$ is a formal power series, $\hat t(0)=\hat T_{11}(0)e^{-\lambda_1^{(0)}\hat g(0)}$.
	
One can also solve the equation on the sectors $\Omega_\alpha$ (see Proposition~\ref{prop:LT-IFT} in the Appendix) and deduce that $\hat\phi(x),\ \hat t(x)$ are Borel $\nu$-summable in the same directions as is the pair $(\hat T_{11}(x),\hat T_{12}(x))$.
\end{proof}

\begin{proof}[\textbf{Proof of Theorem~\ref{thm-B2}}]~\\[3pt]
\noindent	
\textit{1. Analytic equivalence:}
Let us show that if two LDEs \eqref{eq:LT-irregODE} with non-resonant irregular singularity have their companion systems \eqref{eq:LT-irregsyst} analytically equivalent, than so are the equations.
After an analytic change of $x$ we can assume that the formal invariants $\{\lambda_1(x)\delta_\nu^{-1},\lambda_2(x)\delta_\nu^{-1}\}$ are in the canonical form and are the same for the two systems.
Up to right multiplication of $\hat T(x)$ by a constant diagonal matrix, we can also suppose that their collections of Stokes matrices agree.
Therefore, for a singular direction $\beta$:
\begin{align*}
T_{1j,\beta-}(x)&=T_{1j,\beta+}(x)+ s_{\beta} T_{1i,\beta+}(x)e^{\int(\lambda_i(x)-\lambda_j(x))\delta_\nu^{-1}},
&& T_{1i,\beta-}(x)&=T_{1i,\beta+}(x),\\
\tilde T_{1j,\beta-}(x)&=\tilde T_{1j,\beta+}(x)+ s_{\beta}\tilde T_{1i,\beta+}(x)e^{\int(\lambda_i(x)-\lambda_j(x))\delta_\nu^{-1}},
&& \tilde T_{1i,\beta-}(x)&=\tilde T_{1i,\beta+}(x),\\
\end{align*}
for the corresponding Stokes multiplier $s_\beta$ on the position $(j,i)$ of $St_\beta$, $(j,i)=(1,2)$ or $(2,1)$ 
depending on $\beta$ such that $e^{\int(\lambda_i(x)-\lambda_j(x))}$ is flat when $x\to 0$, $\arg x=\beta$.
The conjugation equations to solve are
$$T_{1l,\beta\pm}(x)e^{\int \lambda_l(x)\delta_\nu^{-1}}=t_{\beta\pm}(x)\cdot\big(\tilde T_{1l,\beta\pm} e^{\int \lambda_l\delta_\nu^{-1}}\big)\circ\phi_{\beta\pm}(x),\quad l=1,2.$$
Comparing the above expressions we see that on the intersection sector $\Omega_{\beta+\frac{\pi}{2\nu}}\cap \Omega_{\beta-\frac{\pi}{2\nu}}$ both $\phi_{\beta-}(x)$ and $\phi_{\beta+}(x)$ solve the same functional equation
$$\frac{T_{1j,\beta-}}{T_{1i,\beta\pm}}e^{\int \lambda_j-\lambda_i\delta_\nu^{-1}}=
\frac{\tilde T_{1j,\beta-}}{\tilde T_{1i,\beta\pm}}e^{\int \lambda_j-\lambda_i\delta_\nu^{-1}}\circ\phi_{\beta\pm},$$
hence they will be equal if existence and unicity of a bounded sectoral solution is ensured, and so they will glue up to form an analytic germ $\phi(x)$. And similarly with $t_{\beta+}=t_{\beta-}$.
Writing $\phi_{\beta\pm}(x)=x(1+x^{\nu}g_{\beta\pm}(x))$,
the conjugation equation becomes
\begin{align*}
\log(\tfrac{T_{1j,\beta\pm}}{T_{1i,\beta\pm}})(x)=
\log(\tfrac{\tilde T_{1j,\beta\pm}}{\tilde T_{1i,\beta\pm}})\!\left(x+x^{\nu+1}g_{\beta\pm})\right)+\tfrac{1}{\nu x^\nu}-\tfrac{1}{\nu x^\nu(1+x^{\nu}g_{\beta\pm})^\nu}+ \mu\log(1+x^{\nu}g_{\beta\pm}(x)),
\end{align*}
which by virtue of Proposition~\ref{prop:LT-IFT} in the Appendix, has a unique bounded analytic solution $g_{\beta\pm}(x)$ on $\Omega_{\beta\pm\frac{\pi}{2\nu}}$ satisfying
$g_{\beta\pm}(0)=\log\left(\frac{T_{1j}(0)\tilde T_{1i}(0)}{T_{1i}(0)\tilde T_{1j}(0)}\right)$.

\medskip\noindent
\textit{2. Realization:}
Given $\nu>1$, formal invariants $\lambda_1(x),\lambda_2(x)$, and a collection of projective Stokes operators, we want to show that there exists an equation
\eqref{eq:LT-irregODE} of which they are analytic invariants.

By the Birkhoff--Malgrange--Sibuya theorem (\cite{Bir2}, \cite{Mal1}, \cite{Sib}, \cite[Theorem 20.22]{IlYa}, \cite[Theorem 4.5.1]{BaVa}) there exists a differential system $\delta_\nu v=A(x)v$ with the given formal invariants realizing the associated Stokes matrices as its analytic invariants. Under the non-resonance condition one may as well assume 
that $A(0)=\begin{pmatrix} \lambda_1(0) & 1 \\ 0 & \lambda_2(0)\end{pmatrix}$. 
Writing $A=(A_{ij})$, then the system satisfied by $\tilde v=\begin{pmatrix} 1 & 0 \\ A_{11}(x) & A_{12}(x)\end{pmatrix}v$ is of the companion form.

For the sake of completeness let us give a rough sketch of an alternative proof following the ideas of Malgrange \cite{Mal1, Mal2}.
Let $f(x)=e^{\int (\lambda_2-\lambda_1)\delta_\nu^{-1}(x)}$. For an anti-Stokes direction $\beta$ and a projective Stokes operator $\sigma_\beta$, let $\psi_{\beta}(x)\sim x+O(e^{-\frac{c}{|x|^\nu}})$ (with some $c>0$) be defined by solving the equation $\sigma_\beta\circ f(x)=f\circ\psi_\beta(x)$ on a small sector bisected by $\beta$,
$x\in \Omega_{\beta+\frac{\pi}{2\nu}}\cap\Omega_{\beta-\frac{\pi}{2\nu}}$, where 
$\Omega_\alpha=\{|\arg x-\alpha|<\frac{\pi}{\nu}-\eta,\ |x|<\rho\}$ for some $0<\eta<\frac{\pi}{2\nu}$, $|\rho|>0$.
We want to find bounded analytic sectoral maps $\phi_\alpha(x)=x+O(x^2)$, $x\in \Omega_\alpha$  that solve the cohomological equation
\begin{equation}\label{eq:LT-cohom}
\psi_\beta=\phi_{\beta-\frac{\pi}{2\nu}}\circ\phi_{\beta+\frac{\pi}{2\nu}}^{\circ-1}.
\end{equation}
Then $f\circ\phi_{\beta-\frac{\pi}{2\nu}}=\sigma_\beta\circ(f\circ\phi_{\beta+\frac{\pi}{2\nu}})$, and therefore 
$\Delta(x)=-2\Cal S_\nu(f\circ\phi_\alpha)$ for $x\in\Omega_\alpha$ glue up to an analytic germ on a neighborhood of 0, and the equation \eqref{eq:LT-irregODE} with $p(x)=\lambda_1(x)+\lambda_2(x)$  and $q(x)=\frac{1}{4}\big(\Delta(x)-p(x)^2+2\delta_\nu p(x)\big)$ realizes the invariants.
The problem of solving the cohomological equation \eqref{eq:LT-cohom} can be easily solved in the $\Cal C^\infty$-smooth category by some sectoral germs $\varphi_\alpha$ exponentially tangent to identity (cf. \cite[\S 4.3]{BaVa}). One then can obtain the bounded analytic solutions $\phi_\alpha(x)$
after correcting $\varphi_\alpha(x)$ by a $\Cal C^\infty$ germ $g(x)$,
$\phi_\alpha=\varphi_\alpha\circ g^{\circ-1}$,
obtained as a solution to the Beltrami equation $\frac{\partial_{\bar x}g}{\partial_{x}g}=h$ on a neighborhood of the origin,
where $h(x)$ is an exponentially flat $\Cal C^\infty$ germ defined by 
$h:=\frac{\partial}{\partial\bar x}\varphi_\alpha\big/\frac{\partial}{\partial x}\varphi_\alpha$ which is independent of the sector $\Omega_\alpha$ since
$$0=\tfrac{\partial}{\partial\bar x}\tilde\psi_{\beta}=
\tfrac{\partial}{\partial\bar x}\left(\varphi_{\beta-\frac{\pi}{2\nu}}\circ\varphi_{\beta+\frac{\pi}{2\nu}}^{\circ-1}\right)=
\left(\tfrac{\partial}{\partial\bar x}\varphi_{\beta-\frac{\pi}{2\nu}}-
\tfrac{\partial}{\partial x}\varphi_{\beta-\frac{\pi}{2\nu}}\cdot
\frac{\frac{\partial}{\partial\bar x}\varphi_{\beta+\frac{\pi}{2\nu}}}{\frac{\partial}{\partial x}\varphi_{\beta+\frac{\pi}{2\nu}}}
\right)\cdot
\tfrac{\partial}{\partial\bar x}(\bar \varphi_{\beta+\frac{\pi}{2\nu}}^{-1})$$
by the chain rule for the Wirtinger derivative.
\end{proof}

\begin{remark}
In the formal equivalence problem for non-resonant irregular singularities one can assume that
$\hat T_{11}(0)=\hat T_{12}(0)=1$, hence that $\hat g(0)=0$, i.e. $\hat\phi(x)=x+O(x^{\nu+2})$.
On the other hand, in the analytic equivalence problem one may have $\frac{\hat T_{12}(0)}{\hat T_{11}(0)}\neq1$ and $\hat g(0)\neq 0$, 
 i.e. $\hat\phi(x)=x+O(x^{\nu+1})$. To solve the equation for $\hat t(x)$, one needs that  $\hat\phi(x)=x+O(x^{\nu+1})$.
This observation is at the heart of the following two examples. 
\end{remark}

\begin{example}(Formally equivalent but analytically non-equivalent resonant irregular singularities with analytically equivalent companion systems.)\label{example:LT-1}

Consider the reducible equation of Poincar\'e rank $\nu=2$
$$\big(\delta_2-\alpha_2(x)\big)\big(\delta_2-\alpha_1(x)\big)y=0,$$
with 
$$\alpha_1(x)=1,\qquad \alpha_2(x)=1+x+x^2+\frac{cx^3}{1+cx},\qquad c\in\C.$$
Its basis of solutions is
\begin{align*}
y_1(x)&=e^{-\frac{1}{2x^2}},\quad x\in\C,\\
y_2(x)&=e^{-\frac{1}{2x^2}}\int_{0+}^x e^{-\frac{1}{t}}\tfrac{1+ct}{t^2}dt
=e^{-\frac{1}{2x^2}}\int_{\frac{1}{x}}^{+\infty} e^{-s}(1+\tfrac{c}{s})ds,\quad x\in\C\setminus\R_{\leq 0},\end{align*}
where the integration path in the $s$-variable follows horizontal rays.
The projective Stokes matrices of the associated companion systems are easily calculated using the residue to be $St_0=\id$ and 
$St_\pi=\left(\begin{smallmatrix} 1 & 2\pi ic\\ 0 & 1 \end{smallmatrix}\right)$, which are conjugated by diagonnal matrices for all $c\neq 0$.
Up to analytic gauge transformation, the companion systems can be written as 
$$\delta_2 v=\left(\begin{smallmatrix} \alpha_1(x) & 1\\ 0 & \alpha_2(x) \end{smallmatrix}\right)v,$$
and are all formally equivalent to each other for all $c\in\C$ by a gauge transformation fixing $x$.
Indeed, writing the formal gauge transformation between two such systems with $c$ and $\tilde c$ as $v=\left(\begin{smallmatrix} 1 & \hat f(x)\\ 0 & \frac{1+cx}{1+\tilde cx} \end{smallmatrix}\right)\tilde v$,
then $\hat f(x)$ must satisfy
$$\delta_1\hat f=\tfrac{c-\tilde c}{1+\tilde cx}-(1+x+\tfrac{\tilde cx^2}{1+\tilde cx})\hat f,$$
which has a unique formal solution with $\hat f(0)=c-\tilde c$.
Hence the companion systems are all analytically equivalent for $c,\tilde c\neq 0$.

Also the equations are formally equivalent for all $c\in\C$:
indeed, it suffices to solve the conjugation equations
$$e^{-\frac{1}{x}}\tfrac{1+cx}{x^2}=e^{-\frac{1}{\phi}}\tfrac{1+\tilde c\phi}{\phi^2}\cdot\tfrac{d\phi}{dx},\qquad e^{-\frac{1}{2x^2}}=\hat t(x)e^{-\frac{1}{2\phi^2}}.$$
Writing $\phi(x)=x+x^2g(x)$, then $g(x)$ satisfies an analytic ODE
$$\delta_1 g=e^{-\frac{g}{1+xg}}(1+xg)^2\tfrac{1+cx}{1+\tilde cx(1+xg)}-1-2xg,$$
with a ``saddle-node'' type singularity at $(x,g)=(0,0)$, which is known to have a unique formal solution $\hat g(x)=(c-\tilde c)x+O(x^2)$. Then $\hat t(x)=-\frac{c-\tilde c}{2}+O(x)$ can also be formally solved.

On the other hand, in the problem of analytic equivalence one needs to equalize the bases $c y_1(x),y_2(x)$ and
$\tilde c\tilde y_1(\tilde x),\tilde y_2(\tilde x),$ with respect to which the Stokes matrices agree, and the conjugation equations
$$e^{-\frac{1}{x}}\tfrac{1+c x}{c x^2}
=e^{-\frac{1}{\phi}}\tfrac{1+\tilde c\phi}{\tilde c\phi^2}\cdot\tfrac{d\phi}{dx},\qquad 
ce^{-\frac{1}{2x^2}}=\tilde c\hat t(x)e^{-\frac{1}{2\phi^2}},$$
are solved by the same $\hat\phi(x)$, but this time have no formal solution $t(x)$ if $c\neq\tilde c$, and hence there is no analytic one either.
\end{example}

\begin{proof}[\textbf{Proof of Proposition \ref{prop:LT-reducible}}]~\\[3pt]
	\noindent
	\textit{1.} Indeed, for a reducible LDE \eqref{eq:LT-reducible} one has $\Delta(x)=r(x)^2-2\delta_\nu r(x)$ for $r(x)=\alpha_2(x)-\alpha_1(x)$,
	%$p(x)=\alpha2(x)+\alpha_1(x)$, $\Delta(x)=(\alpha_2(x)-\alpha_1(x))^2-2\delta_\nu(\alpha_2(x)-\alpha_1(x))$ 
	and moreover $y_1(x)=e^{\int \alpha_1(x)\delta_\nu^{-1}}$ is convergent solution. %, $y_1(e^{2pi i}x)=y_1(x)\cdot e^{2\pi i \alpha_1^{(\nu)}}$.
	Vice-versa, if $\Delta(x)=r(x)^2-2\delta_\nu r(x)$ for an analytic $r(x)$, then the equation factors as \eqref{eq:LT-reducible} with 
	$\alpha_2(x)=\frac{1}{2}\big(p(x)+r(x)\big)$ and $\alpha_1(x)=\frac{1}{2}\big(p(x)-r(x)\big)$.
	If $y(x)$ is a ``convergent solution'', then the equation factors as \eqref{eq:LT-reducible} with $\alpha_1(x)=\delta_\nu\log y(x)$ and $\alpha_2(x)=p(x)-\alpha_1(x)$, i.e. $r(x)=p(x)-2\alpha_1(x)$.
	Finally, it is known that the companion system is analytically reducible, i.e. analytically equivalent to one in a triangular form, if and only if the Stokes matrices of  are either all upper triangular or all lower triangular (indeed the formal diagonalizing transformation for a triangular system is triangular and therefore the Stokes matrices will have the same triangular form, and vice versa, a solution to the sectoral cohomological equation with triangular Stokes matrices exists that is triangular). 
	The system $\delta_\nu v=A(x)v$, $A(x)=\big(a_{ij}(x)\big)$, realizing triangular Stokes data can be assumed upper triangular and with $a_{12}(0)=1$ (since $a_{11}(0)\neq a_{22}(0)$), and is therefore it is conjugated to  
	$\delta_\nu v=\begin{pmatrix}\alpha_1(x)&1\\0&\alpha_2(x)\end{pmatrix}v$ for some $\alpha_1(x),\alpha_2(x)$, 
	hence analytically equivalent to the companion system of a reducible LDE \eqref{eq:LT-reducible}.
	
	\medskip\noindent
	\textit{2.} The normal form LDE \eqref{eq:LT-irregnormalform} has two linearly independent solutions $y_j(x)=e^{\int\lambda_j(x)\delta_\nu^(-1)}$, $j=1,2$, the equation \eqref{eq:LT-symmetricpower} has an analytic solution
	$h(x)=\frac{1}{\lambda_2(x)-\lambda_1(x)}$. 
	In general, if $y_j(x)=t_j(x)e^{\int\lambda_j(x)\delta_\nu^(-1)}$, $j=1,2$, is a pair of linearly independent solutions of the LDE, $f(x)=\frac{y_2(x)}{y_1(x)}$, then 
	$h(x)=\frac{f}{\delta_\nu f}=\frac{1}{\lambda_2(x)-\lambda_1(x)+\delta_\nu\log\frac{t_2}{t_1}}$ is an analytic solution to
	\eqref{eq:LT-symmetricpower}.
	
	Supposing that $h(x)$ is a nontrivial analytic solution to \eqref{eq:LT-symmetricpower}, and let us show that the LDE is reducible. 
	First, let us notice that form \eqref{eq:LT-symmetricpower} it follows that $h(x)^2\Delta(x)=c^2+O(x^{2\nu+1})$ for some $0\neq c\in\C$.
	Let $r(x,c):=\frac{c-\delta_\nu h(x)}{h(x)}$ and $\Delta(x,c^2):=r_c(x)^2-2\delta_\nu r_c(x)=\frac{c^2-\big(\delta_\nu h(x)\big)^2+2h(x)\delta_\nu^2 h(x)}{h(x)^2}$, $c\in\C$.
	Considering the equation \eqref{eq:LT-symmetricpower} as a non-homogeneous first order linear differential equation for unknown $\Delta$ with coefficients determined by $h(x)$, then $\Delta(x,a^2)$ are its solutions for all $a\in\C$ and for the reason of dimension there are no other solutions. Hence $\Delta(x)=\Delta(x,c^2)$ and the LDE is reducible, with 
	$r(x)=r(x,\pm c)$, two different solution to $\Delta(x)=r(x)^2-2\delta_\nu r(x)$ \eqref{eq:LT-ricatti},
	and $y_1(x)=e^{\int \alpha_1(x)\delta_\nu^{-1}}$, $y_2(x)=h(x)e^{\int \alpha_2(x)\delta_\nu^{-1}}$,
	where $\alpha_1(x)=\frac{1}{2}\big(p(x)-r(x)\big)$ and $\alpha_2(x)=\frac{1}{2}\big(p(x)+r(x)\big)$
	are two linearly independent convergent solutions to the LDE. 
	
	Vice-versa, if $r_j(x)$, $j=1,2$, are two different solution to \eqref{eq:LT-ricatti},
	then $h(x)=\frac{1}{r_2(x)-r_1(x)}$, $\delta_\nu\log h(x)=-\frac{r_1(x)+r_2(x)}{2}$, is an analytic solution to \eqref{eq:LT-symmetricpower}.
	
	Finally, a system with an irreducible irregular singularity and formal invariants $\{\lambda_1(x)\delta_\nu^{-1},\lambda_2(x)\delta_\nu^{-1}\}$ is analytically equivalent to
	$\delta_\nu v=\begin{pmatrix}\lambda_1(x)&0\\0&\lambda_2(x)\end{pmatrix}v$,
	if and only if its collection of Stokes matrices is trivial, and one concludes by Theorem~\ref{thm-B2}.
\end{proof}

\begin{proof}[\textbf{Proof of Theorem~\ref{thm-D}}]
	Two $2\times 2$ linear differential systems with non-resonant irregular singularity at the origin of Poincar\'e rank $\nu=1$ are analytically equivalent if and only if they have the same formal invariants and the same trace of monodromy.
	Indeed, the monodromy matrix of such system with respect to the sectoral fundamental solution $V_{-\frac{\pi}{2}}(x)$ is of the form 
	$M=\left(\begin{smallmatrix} e^{2\pi i \lambda_1^{(1)}} & 0 \\ 0 & e^{2\pi i \lambda_2^{(1)}} \end{smallmatrix}\right)
	\left(\begin{smallmatrix} 1 & 0 \\[3pt] s_\pi & 1 \end{smallmatrix}\right)
	\left(\begin{smallmatrix} 1 & s_0 \\[3pt] 0 & 1 \end{smallmatrix}\right),$
	where the product $s_0s_\pi\neq 0$ is the invariant characterizing the pair of Stokes matrices.
	Therefore $\det M=e^{2\pi i (\lambda_1^{(1)}+\lambda_2^{(1)})}=e^{2\pi i p^{(1)}}$ and the quantity
	$$(\det M)^{-\frac{1}{2}}\tr M=e^{\pi i (\lambda_2^{(1)}-\lambda_1^{(1)}}(1+s_0s_\pi)+e^{\pi i(\lambda_1^{(1)}-\lambda_2^{(1)}}=
	2\cos \pi\mu+e^{\pi i\mu}s_0s_\pi$$
	is a natural invariant.
	On the other hand the eigenvalues of the residue matrix of the companion system at $x=\infty$ are
	$-\frac{p^{(1)}-1}{2}\pm\frac{1}{2}\sqrt{\Delta^{(2)}+1}$, where $\Delta^{(2)}=(p^{(1)})^2+4q^{(2)}-2p^{(1)}$,
	and therefore
	$$(\det M)^{-\frac{1}{2}}\tr M=2\cos\pi\sqrt{\Delta^{(2)}+1},$$
	from which one has (cf. \cite[p. 144]{MR})
	$$s_0s_\pi=4e^{-\pi i\mu}\sin\pi\big(\frac{\mu+\sqrt{\Delta^{(2)}+1}}{2}\big)\sin\pi\big(\frac{\mu-\sqrt{\Delta^{(2)}+1}}{2}\big).$$
	For given formal and analytic invariants, the equation for $\Delta^{(2)}$ can be always solved.

	If the LDE is reducible, written as \eqref{eq:LT-reducible},
	then there are only two analytic equivalence classes within a given formal equivalence class: one corresponds to both Stokes matrices trivial, and the other to one non-trivial Stokes matrix conjugated to $\left(\begin{smallmatrix} 1 & 1 \\[3pt] 0 & 1 \end{smallmatrix}\right)$.
	The equation has two canonical solutions: 
	\begin{align*}
	y_1(x)&=e^{\int \alpha_1(x)\delta_1^{-1}},\\
	y_2(x)&=e^{\int \alpha_1(x)\delta_1^{-1}}\int_{0}^{x}e^{\int \alpha_2(x)-\alpha_1(x)\delta_1^{-1}}\delta_1^{-1},
	\end{align*}
	$x\in\C\sminus(\alpha_2^{(0)}-\alpha_1^{(0)})\R_{\leq 0}$,  where the integration path follows a real trajectory of $\frac{\delta_1}{\alpha_2(x)-\alpha_1(x)}$.
	Assuming that $\alpha_2(x)-\alpha_1(x)=1+\mu x+x^2r(x)$ for some analytic germ $r(x)$,
	then 
	$$y_2(x)=y_1(x)\int_{0+}^{x}e^{-\frac{1}{x}}x^{\mu-2} R(x)dx=y_1(x)\int_{\frac{1}{x}}^{+\infty}e^{-s}s^{-\mu} R(\tfrac{1}{s})ds,$$
	where $R(x):=e^{\int r(x)dx}$, the integration in the second integral following a horizontal ray.
	Denoting $y_{2,+}(x)$, resp. $y_{2,-}(x)$, the branch of $y_2(x)$ on $\arg x\in]\pi,2\pi[$, resp. $\arg x\in]0,\pi[$, then for $\arg x=\pi$
	\begin{equation}\label{eq:LT-w+-}
	y_{2,+}(x)-y_{2,-}(x)=y_1(x)s_\pi,\qquad s_\pi=\big(e^{2\pi i\mu}-1\big) \int_{0}^{+\infty}e^{-s}{s}^{-\mu} R(\tfrac{1}{s})ds. 
	\end{equation}
	In particular, if $\alpha_2(x)-\alpha_1(x)=1+\mu x$, i.e. $R(x)=1$, then  $s_\pi=\big(e^{2\pi i\mu}-1\big)\Gamma(1-\mu)= e^{\pi i\mu}\frac{2\pi i}{\Gamma(\mu)}$  which vanishes if and only if $\mu\in\Z_{\leq 0}$.
	And if $\alpha_2(x)-\alpha _1(x)=1+\mu x-x^{2}$, i.e. $R(x)=e^{-x}$, then   
	$s_\pi=\big(e^{2\pi i\mu}-1\big)\sum_{j\geq0}(-1)^{j}\Gamma(1-\mu-j)=2\pi ie^{\pi i\mu}\sum_{j\geq 0}\frac{1}{\Gamma(\mu+j)}$ which is positive for every  $\mu\in\Z$. 
\end{proof}

\paragraph{Non-degenerate resonant irregular singularities.}
Suppose $J_0^\nu\Delta(x)=x$ and let $J_0^\nu p(x)=P(x)$, $P(0)$.
The LDE may be rewritten as
$$\delta_{\nu-\frac12}^2y-(\tfrac{p(x)}{x^{\frac12}}-\tfrac12 x^{\nu-\frac12})\delta_{\nu-\frac12}y-\tfrac{q(x)}{x}y=0,$$
which is non-resonant in the variable $x^{\frac12}$.
Correspondingly it has a formal solution basis
\begin{equation}\label{eq:LT-irregsol1}
\hat y_1(x)=\hat T_{11}(x^{\frac12})x^{-\frac14}e^{\int\frac12(P(x)-x^{\frac12})\delta_\nu^{-1}},\qquad 
\hat y_2(x)=\hat T_{12}(x^{\frac12})x^{-\frac14}e^{\int\frac12(P(x)+x^{\frac12})\delta_\nu^{-1}}, 
\end{equation}
where $\hat T_{11}(0)=\hat T_{12}(0)=1$. 
By the results of the previous section, for every two such LDE for $y$ and $\tilde y$ respectively there exists a unique formal transformation in $x^{\frac12}$
$$\tilde x^{\frac12}=\hat\phi^{\frac12}(x^{\frac12})=x^{\frac12}+O(x^{\nu+\frac12}),\quad
\tilde y=\hat t(x^{\frac12})y,\ \hat t(0)=1,$$
that transforms the respective solutions \eqref{eq:LT-irregsol1} one to the other,
and is Borel $2\nu-1$-summable in the variable $x^{\frac12}$.

\begin{proof}[\textbf{Proof of Proposition~\ref{prop:LT-formalres}}]
Let us show that there exists a formal transformation in $x$
$$\tilde x=\hat\phi(x)=x+O(x^{\nu+1}),\quad \tilde y=\hat t(x)y,\ \hat t(0)=1,$$
between the LDEs. This transformation will necessarily also transform the respective solutions \eqref{eq:LT-irregsol1} one to the other, so by the unicity it will agree with the above one.
We can assume $\tilde\Delta(\tilde x)=\tilde x$, $\Delta(x)=x+O(x^{\nu+1})$, and we construct the transformation $\tilde x=\hat\phi(x)$ as a formal infinite composition $\hat\phi=\ldots\phi_{\nu+2}\circ\phi_{\nu+1}\circ\phi_{\nu}$,
where $\phi_k(x)=x+a_kx^{k+1}$ is such that if $\Delta_k=x+b_kx^{k+1}+O(x^{k+2})$ and $\tilde x=\phi_k(x)$ then $\Delta_{k+1}(\tilde x)=\tilde x+O(\tilde x^{k+2})$.
Plugging this into \eqref{eq:LT-transformationDelta} gives
$x+a_kx^{k+1}+O(x^k+2)=x+(2k-2\nu+1)b_kx^{k+1}+O(x^k+2)$, hence $b_k=\frac{a_k}{2k-2\nu+1}$ is uniquely determined.
If $\hat\phi(x)=x+O(x^{\nu+1})$ and $p(x)=\tilde p(x)+O(x^{\nu+1})$, then also the equation \eqref{eq:LT-transformationp} for $\log\hat t(x)$ has a unique formal solution.
\end{proof}
 
\begin{proof}[\textbf{Proof of Theorem~\ref{thm-B3}}]
If the Stokes operators agree, then by the reasoning of the proof of Theorem~\ref{thm-B2}, the sectoral transformation
between the respective solutions \eqref{eq:LT-irregsol1} of the two LDEs glue up to an analytic transformation in the variable
$x^{\frac12}$
$$\tilde x^{\frac12}=\phi^{\frac12}(x^{\frac12})=x^{\frac12}+O(x^{\nu+\frac12}),\quad
\tilde y=t(x^{\frac12})y,\ \hat t(0)=1.$$
The Taylor expansion of this transformation is a formal transformation that is tangent to identity,
hence by Proposition~\ref{prop:LT-formalres} $\phi(x^{\frac12})$ and $t(x^{\frac12})$ contain only whole powers of $x$,
therefore is analytic in $x$.
\end{proof}

\begin{proof}[\textbf{Proof of Theorem~\ref{thm-D2}}]
By the same reasoning as in the proof of Theorem~\ref{thm-D2}.
\end{proof}

\begin{proof}[\textbf{Proof of Theorem~\ref{thm-E}}]
Up to a composition with transformation $y\mapsto x^m y$, which changes $J_0^\nu p(x)\cdot\delta_\nu^{-1}$ by $m\delta_0^{-1}$ while preserving the monodromy and the Stokes operators,
the meromorphic classification problem is reduced to the analytic one. 	
In fact, in the irregular case meromorphic transformations preserve the canonical sectoral solution bases. 
\end{proof}

\begin{proof}[\textbf{Proof of Theorem~\ref{thm-F}}]
Calculating the coefficients of $\partial_x$ and $\partial_y$ in \eqref{eq:LT-Lie} with \eqref{eq:LT-LieX} and \eqref{eq:LT-LieY} both give $\alpha=\frac{d}{dx}g$.
Then the coefficients of $y\partial_{y_x}$ and  $y_x\partial_{y_x}$ in \eqref{eq:LT-Lie} give respectively
\begin{align*}
 2\tfrac{q}{x^{2\nu+2}}\tfrac{d}{dx}g+g\tfrac{d}{dx}\!\left(\tfrac{q}{x^{2\nu+2}}\right)&=
 \tfrac{d^2}{dx^2}f-\left(\tfrac{p}{x^{\nu+1}}-\tfrac{\nu+1}{x}\right)\tfrac{d}{dx}f\\
 \tfrac{d^2}{dx^2}g+\left(\tfrac{p}{x^{\nu+1}}-\tfrac{\nu+1}{x}\right)\tfrac{d}{dx}g+g\tfrac{d}{dx}\!\left(\tfrac{p}{x^{\nu+1}}-
 \tfrac{\nu+1}{x}\right)&=2\tfrac{d}{dx}f.
\end{align*}
Hence 
\begin{align*}
2f(x)=\tfrac{d}{dx}g+\left(\tfrac{p(x)}{x^{\nu+1}}-\tfrac{\nu+1}{x}\right)g(x)+2f^{(0)}, \qquad\text{for some } f^{(0)}\in\C.
\end{align*}
Let $h(x)=\frac{g(x)}{x^{\nu+1}}$, then
%\begin{align*}
$2f=\delta_\nu h+p\,h+2f^{(0)},$ and $\delta_\nu^2f-p\,\delta_\nu f=2q\,\delta_\nu h+\delta_\nu q\,h$,
%\end{align*}
from which
\begin{equation*}
 \delta_\nu^3 h-\Delta\,\delta_\nu h-\tfrac{1}{2}\delta_\nu\Delta\,h=0.
\end{equation*}

For a strongly non-resonant regular singularity \eqref{eq:LT-B1a}, the equation to solve is
$$\delta_0^3 h-(\lambda_1-\lambda_2)^2\delta_0 h=0,$$
which has a basis of solutions $h_1(x)=1$, $h_2(x)=x^{\lambda_1-\lambda_2}$, $h_3(x)=x^{\lambda_2-\lambda_1}$ if $\lambda_1\neq\lambda_2$, and
$h_1(x)=1$, $h_2(x)=\log x$, $h_3(x)=(\log x)^2$ if $\lambda_1=\lambda_2$. 

For a resonant regular singularity \eqref{eq:LT-B1c}, the equation for $h$ is of the form \eqref{eq:LT-symmetricfactorized} 
with $\nu=0$ and $r(x)=-\frac{k}{1-x^k}$, 
which has a basis of solutions $h_1(x)=\frac{x^k}{1-x^k}$, $h_2(x)=\frac{1}{1-x^k}+\frac{kx^k}{1-x^k}\log x$, and $h_3(x)$ 
given by $\big(\delta_0+r(x)\big)h_3(x)=\frac{1}{x^k}+k\log x$.

If $\nu>0$ and $\Delta(0)\neq 0$, then by Proposition~\ref{prop:LT-reducible} the equation \eqref{eq:LT-symmetricpower} has a non-trivial analytic solution $h(x)$ if and only if the LDE is analytically equivalent to its formal normal form \eqref{eq:LT-irregnormalform}, for which one has $h(x)=\frac{c}{\lambda_2(x)-\lambda_1(x)}$, $c\in\C$.
\end{proof}

\section*{Appendix: Implicit function theorem for Borel summable power series}

There are several equivalent ways to define Borel summability (see e.g. \cite{Ba,Mal,MR,Lod}). We will use the following one due to J.-P.~Ramis.

Let $(E,\|\cdot\|)$ be a Banach space: we will consider the following two, (i) the field $(\C,|\cdot|)$, (ii) the space of bounded analytic functions on some small disc $D=\{|y|\leq\epsilon\}$, $\epsilon>0$, together with supremum norm.

Let $\hat f(x)=\sum_{n\geq 0} f^{(n)}x^n\in E\llbracket x\rrbracket$ be a formal power series with coefficients $f^{(n)}\in E$.
\begin{itemize}
\item An open \emph{sectoral domain} at the origin is a simply connected domain $U$ in $\C$ (or in the Riemann surface of logarithm) with $0$ in its boundary that can be written as an (infinite) union of open sectors at 0 of increasing angular opening (and decreasing radius).
\item Let $U$ be an open sectoral domain at the origin. An analytic function $f:U\to E$ is said to be \emph{$s$-Gevrey asymptotic} to $\hat f$, $s>0$, if for every sector $V\subset\subset U$ (i.e. such that $\overline{V}\subset U\cup\{0\}$) there exist $C,A>0$ such that
$$\|f(x)-\sum_{n=0}^{N-1} f^{(n)}x^n \|\leq C|x|^NA^N\Gamma(1+sN),\qquad\text{for all $N>0$ and all $x\in V$}.$$
\item An analytic function $f:U\to E$ is said to be \emph{exponentially flat of order $\nu>0$} if it is $\frac{1}{\nu}$-Gevrey asymptotic to the zero series. This is equivalent to ask that for every sector $V\subset\subset U$ there exist $C,A>0$ such that
$$\|f(x)\|\leq Ce^{-\frac{A}{|x|^\nu}},\qquad\text{for all}\ x\in V.$$
\item The formal power series $\hat f(x)$ is said to be \emph{Borel $\nu$-summable in a direction} $\alpha\in\R$ for $\nu>0$, if there exists an open sector
$U$ of angular width $>\frac{\pi}{\nu}$ bisected by $\alpha$, and a function $f_\alpha:U\to E$ that is $\frac{1}{\nu}$-Gevrey asymptotic to $\hat f(x)$.
It is said to be \emph{Borel $\nu$-summable} if it is summable in all directions $\alpha\in [0,2\pi[$ up to finitely many.  
The directions of non-summability are called \emph{anti-Stokes} or \emph{singular}.
\end{itemize}

Let $\beta_0<\ldots <\beta_{m-1}$ be the anti-Stokes directions of a Borel $\nu$-summable $\hat f(x)$ in the interval $[0,2\pi[$, $\beta_m=\beta_0+2\pi$.
For $\eta>0$ arbitrarily small, let 
\begin{equation}\label{eq:LT-covering}
V_j=\{\arg x\in]\beta_j-\tfrac{\pi}{2\nu}+\eta,\beta_{j+1}+\tfrac{\pi}{2\nu}-\eta[,\ |x|<\rho(\eta)\},\quad j=0,\ldots,m-1,
\end{equation} 
with some $\rho(\eta)>0$,
be a cyclic covering of a pointed neighborhood of $0$ by sectors, and $f_{V_j}:V_j\to E$ the Borel sum of $\hat f$ in the directions $\alpha\in]\beta_j+\eta,\beta_{j+1}-\eta[$.
Then by definition $f_{V_{j+1}}-f_{V_j}$ is exponentially flat of order $\nu$ on the intersection $V_j\cup V_{j+1}$.

By the Ramis--Sibuya theorem the converse is also true, giving thus a useful characterization of Borel summability.

\begin{theorem*}[Ramis--Sibuya]
Let $V_j$, $j\in\Z_m$, be a cyclic covering of a pointed neighborhood of $0$ by sectors.
Let $f_{V_j,V_{j+1}}:V_j\cap V_{j+1}\to E$ be exponentially flat of order $\nu$, $j\in\Z_m$.
Then there exists a formal power series $\hat f(x)$, and a ``cochain'' of sectoral functions $f_{V_j}:V_j\to E$, $j\in\Z_m$, that are
$\frac{1}{\nu}$-Gevrey asymptotic to $\hat f(x)$, such that $f_{V_j,V_{j+1}}=f_{V_{j+1}}-f_{V_j}$.

In particular, if the angular opening of $V_j$ is $>\frac{\pi}{\nu}$, then $\hat f(x)$ is Borel $\nu$-summable in the directions covered by 
$e^{-\frac{\pi}{2\nu}}V_j\cap e^{\frac{\pi}{2\nu}}V_j$.
\end{theorem*}

\begin{proposition}\label{prop:LT-IFT}
Let $\hat F(x,y)=\sum_{j\geq 0}F^{(j)}(y)x^j$ be a formal power series of $x$ with coefficients bounded and analytic in $y$ on some small disc $D=\{|y|\leq\epsilon\}$, $\epsilon>0$,
that is Borel $\nu$-summable.
Suppose that $F^{(0)}(0)=0$ and $\frac{dF^{(0)}}{dy}(0)\neq 0$.
Then the implicit equation 
$$\hat F(x,g(x))=0,$$
has a unique formal solution $\hat g(x)=\sum_{j\geq 0}g^{(j)}x^j$, $g^{(0)}=0$, which is Borel $\nu$-summable with singular directions among the   singular directions of $\hat F(x,y)$.  
\end{proposition}

\begin{proof}[\textbf{Proof.}]
For any small $\eta>0$, let $F_{V_j}:V_j\times D\to\C$, $j\in\Z_m$, be the ``cochain'' of the Borel sums of $\hat F$ on the cyclic sectoral covering $V_j$ \eqref{eq:LT-covering}.
Then there exist $C,A>0$ such that
$$|F_{V_{j+1}}(x,y)-F_{V_j}(x,y)|\leq C e^{-\frac{A}{|x|^\nu}} \quad\text{on the intersections} \ (x,y)\in (V_j\cap V_{j+1})\times D.$$ 
We shall solve the sectoral implicit equations
$$F_{V_j}(x,g_{V_j}(x))=0,$$
for $x\in V_j$, by replicating the usual proof of the implicit function theorem.
The sectoral solution $g_{V_j}(x)$ is obtained as a fixed point of the operator
$$\Cal K_j:y(x)\mapsto y(x)-\big(\tfrac{\partial F^{(0)}}{\partial y}(0)\big)^{-1}F_{V_j}(x,y(x)),\qquad g_{V_j  }(x)=\lim_{n\to+\infty}\Cal K_j^{\circ n}(x,0),$$
on the space of bounded analytic functions on the sector $V_j$ with $\sup_{x\in V_j} |y(x)|\leq\frac{\epsilon}{2}$ for some small $\epsilon>0$, supposing that the radius of $V_j$ is small enough.
Indeed, let us estimate
$$\big|\tfrac{\partial \Cal K_j}{\partial y}(x,y)\big|\leq 
\big|\tfrac{\partial F^{(0)}}{\partial y}(0)\big|^{-1} \big|\tfrac{\partial F_{V_j}}{\partial y}(x,y)-\tfrac{\partial F^{(0)}}{\partial y}(y)\big|+
\big|\big(\tfrac{\partial F^{(0)}}{\partial y}(0)\big)^{-1}\tfrac{\partial F^{(0)}}{\partial y}(y)-1\big|,$$
for $x\in V_j$, $|y|<\frac{\epsilon}{2}$.
The second term can be made arbitrarily small by restricting the radius $\epsilon$ of $D$.
The first term can expressed by the Cauchy formula as
$$\frac{1}{2\pi}\big|\tfrac{\partial F^{(0)}}{\partial y}(0)\big|^{-1} 
\left|\int_{|\zeta-y|=\tfrac{\epsilon}{2}}\tfrac{F_{V_j}(x,\zeta)- F^{(0)}(\zeta)}{\zeta-y}d\zeta\right|\leq c|x|,$$
for some $c>0$.
Hence, up to restricting the radius $\rho(\eta)$ of the sector $V_j$, one can assume that $\big|\frac{\partial \Cal K_j}{\partial y}(x,y)\big|\leq\frac{1}{2}$ for $x\in V_j$, $|y|<\frac{\epsilon}{2}$,
and the operator is contractive.

Let us now show that
$|g_{V_{j+1}}(x)-g_{V_j}(x)|\leq K e^{-\frac{A}{|x|^\nu}}$ on the intersections $V_{j+1}\cap V_j$ for some $K>0$ in order to apply the Ramis--Sibuya theorem and obtain the Borel summability.
For $x\in V_{j+1}\cap V_j$ and $|y_{j+1}|,|y_j|\leq\frac{\epsilon}{2}$ one can estimate
$|\Cal K_{j+1}(x,y_{j+1})-\Cal K_j(x,y_j)|\leq |\Cal K_{j+1}(x,y_{j+1})-\Cal K_{j+1}(x,y_j)|+|\Cal K_{j+1}(x,y_j)-\Cal K_j(x,y_j)|$,
where the first term is bounded by 
$|y_{j+1}-y_j|\cdot\int_0^1\big|\frac{\partial \Cal K_j}{\partial y}(x,ty_{j+1}+(1-t)y_j)\big|dt\leq\frac{1}{2}|y_{j+1}-y_j|$,
and the second term is bounded by $|\tfrac{\partial F^{(0)}}{\partial y}(0)|^{-1} C e^{-\frac{A}{|x|^\nu}}$. Therefore the difference of all the respective iterations satisfy $|\Cal K_{j+1}^{\circ n}(x,0)-\Cal  K_j^{\circ n}(x,0)|\leq K e^{-\frac{A}{|x|^\nu}}$ for $K=2|\tfrac{\partial F^{(0)}}{\partial y}(0)|^{-1} C$ and so do the limits $|g_{V_{j+1}}(x)-g_{V_j}(x)|\leq K e^{-\frac{A}{|x|^\nu}}$.
\end{proof}

\section*{Acknowledgement}
The motivation to look at analytic classification of second order linear differential equations came from the article \cite{KS} of Kossovskiy \& Schafikov where the question of analytic point equivalence was posed and a particular case investigated in connection to a classification problem in CR geometry. I want to thank to Ilya Kossovskiy for his generous support during the preparation of this paper.

\bigskip
\goodbreak

\end{document}